\documentclass[reqno]{amsart}
\usepackage{tkz-euclide}
\usepackage{hyperref}
\usepackage{amsmath}
\usepackage{amssymb} 
\usepackage{mathrsfs}
\usepackage{graphicx}
\usepackage{tikz}
\usepackage[font=footnotesize, labelfont=bf]{caption}
\usepackage{subcaption}
\usepackage{xcolor}  
\usepackage{amsthm}
\usepackage{esint}
\usepackage{float}
\usepackage{enumitem}
\usepackage[english]{babel}

\definecolor{LimeGreen}{cmyk}{0.50, 0.5, 1, 0}
\newcommand{\strokedint}{\fint}

\newcommand{\eps}{\varepsilon} 
\newcommand{\dx}{\, {\rm d}x}
\newcommand{\dy}{\, {\rm d}y}
\newcommand{\dz}{\, {\rm d}z}
\newcommand{\ds}{\, {\rm d}s}

\newcommand{\dP}{\, {\rm d}P}
\newcommand{\e}{\varepsilon}

\DeclareMathOperator{\dist}{dist}
\newcommand{\Sph}{{\mathbb S}}
\newcommand{\res}{\mathop{\hbox{\vrule height 7pt width 0.5pt depth 0pt
			\vrule height 0.5pt width 6pt depth 0pt}}\nolimits}

\theoremstyle{plain}
\begingroup
\newtheorem{theorem}{Theorem}[section]
\newtheorem{lemma}[theorem]{Lemma}
\newtheorem{proposition}[theorem]{Proposition}
\newtheorem{corollary}[theorem]{Corollary}
\endgroup

\numberwithin{equation}{section}
\let \O=\Omega
\newcommand{\N}{\mathbb{N}}
\newcommand{\Z}{\mathbb{Z}}
\newcommand{\Q}{\mathbb{Q}}
\newcommand{\R}{\mathbb{R}}

\newcommand{\I}{\mathcal{I}}

\newcommand{\A}{\mathcal{A}}
\renewcommand{\S}{\mathbb{S}}
\renewcommand{\L}{\mathcal{L}}
\renewcommand{\H}{\mathcal{H}}

\newcommand{\dHn}{\, {\rm d}\H^{n-1}}

\newcommand{\T}{\mathcal{T}}
\newcommand{\x}{\times }
\newcommand{\m}{\mathbf{m}}

\newcommand{\defas}{:=}

\newenvironment{step}[1]{\medskip\textit{Step #1:}}{}

\theoremstyle{definition}
\begingroup
\newtheorem{definition}[theorem]{Definition}

\endgroup
\theoremstyle{remark}
\newtheorem{remark}[theorem]{Remark}
\renewcommand{\tilde}{\widetilde}

\renewcommand{\d}{\, \mathrm{d}}
\usepackage[normalem]{ulem}

\title[Non-local approximation of free-discontinuity problems in linear elasticity]
{Non-local approximation of free-discontinuity problems in linear elasticity and application to stochastic homogenisation}

\author[R. Marziani]{Roberta Marziani}
\address[R. Marziani]{Faculty of Mathematics, TU Dortmund University, Vogelpothsweg 87, 44227 Dortmund, Germany.}
\email[R. Marziani]{roberta.marziani@tu-dortmund.de}
\author[F. Solombrino]{Francesco Solombrino}
\address[F. Solombrino]{Dipartimento di Matematica e Applicazioni ``R.~Caccioppoli'', Universit\`a di Napoli Federico II, via Cintia, 80126 Napoli, Italy.}
\email[F. Solombrino]{francesco.solombrino@unina.it}

\begin{document}

\begin{abstract}
We analyse the $\Gamma$-convergence of general non-local convolution type functionals with varying densities depending on the space variable and on the symmetrized gradient. The limit is a local free-discontinuity functional, where the bulk term can be completely characterized in terms of an asymptotic cell formula. From that, we can deduce an homogenisation result in the stochastic setting.
\end{abstract}

\maketitle
{\small
	\noindent 
	
	\medskip
	
	\noindent 
}
\section{Introduction}
This paper is focused on the approximation of brittle fracture energies for linearly elastic materials, by means of {\it non-local functionals} defined on Sobolev spaces. The asymptotic behavior of these functionals will simultaneously show the emergence both of effective energies for the elastic deformation (which may be, e.g., the output of homogenization), and of Griffith-type surface energies accounting for crack formation. In turn, this result can be further generalized to the setting of stochastic homogenization with fracture.

Precisely our results will extend the range of application of the recent papers \cite{FarSciSol, SS21} while also providing some relevant technical improvement. We briefly comment on these previous contributions, in order to introduce our results.
There, an approach originally devised by Braides and Dal Maso~\cite{BraDal97} 
for the approximation of the Mumford-Shah functional has been generalized to the linearly elastic setting. Namely, it was shown that, for a given bounded increasing function $f\colon \R^+\to \R^+$  the energies
\begin{equation}\label{1intro}
\frac1{\e_k}\int_U f\Big(\e_kW(e(u))*\rho_k(x)
\Big)\dx\,,
\end{equation}
$\Gamma$-converge, in the $L^1(U)$-topology, to the functional 
\begin{equation}\label{2intro}
\alpha\int_{U} W(e(u))\,\mathrm{d}x  + \beta \int_{J_u}\phi_\rho(\nu_u)\dHn\,,
\end{equation}
with $\alpha=f^\prime(0)$ and $\beta=\lim_{t\to +\infty}f(t)$. Above, $\rho_k$ are rescaled convolution kernels with unit mass and compact convex symmetrical  support $S$, $\phi_\rho$ is (twice) the support function of $S$ (see \eqref{def:phi} for its precise definition), $W(e(u))$ is a convex elastic energy with superlinear $p$ growth
 depending on the symmetrized gradient $e(u)$ of a  vector-valued displacement $u$, whose  jump set is denoted by $J_u$. Notice that the effective domains of the approximations and of the limit are different. Actually \eqref{1intro}  is finite on the Sobolev space $W^{1,p}(U;\R^n)$, while the energy space of \eqref{2intro} is the one of generalized functions with bounded deformation $GSBD^p(U)$, introduced in \cite{DM2013}.  \\

We stress that the above results allowed one for a general (convex) bulk energy. The proof strategy cannot rely, at least when estimating the bulk part, on any slicing procedure. The latter is instead successful in the particular case  $W(\xi)=|\xi|^p$, considered for instance in \cite{Negri2006}. We also remark that 
 the results of \cite{FarSciSol, SS21} were obtained under an additional structural assumption on the kernels $\rho_k$, which have to be radial with respect to the norm induced by $S$.  In the particular case considered in \cite{Negri2006}, this restriction was instead not needed.\\

A natural extension of the aforementioned models allows one to include
 an explicit dependence on $k$ and on the space variable of the energy density. This amounts to consider functionals of the form 
\begin{equation}\label{3intro}
	\frac1{\e_k}\int_U f\Big(\e_kW_k(\cdot,e(u))*\rho_k(x)
	\Big)\dx\,,
\end{equation}
whose limit behaviour is the object of the present paper. This more general setting is indeed suitable for further applications, if one thinks to the mechanical counterpart of the model. Indeed energy densities of type $W_k(y,M)$, where $y$ is the position 
in the reference configuration, are customary when dealing with heterogeneous material with some microstructures.  The prototypical example is the case of \textit{homogenisation}, that is, when $W_k(y,M)=W(\frac{y}{\delta_k},M)$ with $\delta_k\searrow0$. \\

The main result of our paper is contained in Theorem \ref{thm:main-theorem}. There we show that the functionals in \eqref{3intro} $\Gamma$-converge to a limit energy of the form 
\begin{equation}\label{4intro}
	\alpha\int_{U} W(x,e(u))\,\mathrm{d}x  + \beta \int_{J_u}\phi_\rho(\nu_u)\dHn\,.
\end{equation}
Above, the limit bulk density $W$ can be characterised in terms of cell formula (see \eqref{eq:cell-formula-inf}-\eqref{eq:cell-formula-sup}). Remarkably, that coincides exactly with the asymptotic formula that one would obtain by considering the limit behaviour of the local energies 
$
\int_UW_k(x,e(u))\dx
$
in the Sobolev space $W^{1,p}(U)$. Hence, a decoupling effect between bulk and surface contribution occurs, since the volume energy only depends on $f$ trough its derivative at the origin. A similar effect has been observed in \cite{Cortesani} where the analogue of  \eqref{3intro} for energies depending on the full deformation gradient was taken into account. On the one hand, the possibility of using smooth truncations (a tool which is not available in $GSBD$) allowed the author there to replace $f$ by a sequence $f_k$ and to derive more general surface energies in the limit. On the other hand, the precise characterisation of the volume energy density was obtained at the expense of an additional technical condition on the $W_k$'s (the so-called \textit{stable} $\gamma$-convergence). It actually turns out, as an output of our proof strategy, that this extra assumption can be dropped out (see  Appendix \ref{sec:appendix}). Thus, our results also permit some improvement in the previous literature about non-local approximation of free-discontinuity problems.\\

We now come to the description of our proof technique. The most difficult point is  the lower bound for the bulk contribution. This is done  in Proposition \ref{blow-up}, by means of a localisation and blow-up procedure which contains some elements of novelty in the non-local setting. More precisely we  consider the blow-up of sequences with equi-bounded energies at a Lebesgue point for the limit energy.  
A crucial task is to gain a uniform control on the $L^p$ norm of the symmetrized gradients of the blow-up functions up to sets with vanishing perimeter. This allows us to apply \cite[Lemma 5.1]{FPS} (which relies on the Korn-type inequality of \cite{CCS}): we can substitute, with almost no change in the energy, the above mentioned sequence with a more regular one  bounded in $W^{1,p}$. Exploiting the properties of $f$ we are then reconducted to analyse the limit behaviour on small squares of a local energy in $W^{1,p}$, which can be estimated from below via a cell formula.  \\

An optimal estimate from below for the surface term can be obtained by means of a slicing procedure (Proposition \ref{prop:lb-surf}). As for the $\Gamma$-limsup inequality it can be achieved by a direct construction for a class of competitors with regular jump set, which are dense in energy. Here we use the classical approximation result of \cite{Chamb-Cris2019, Cortesani-Toader-Density}. \\

We underline that even in the case of \eqref{1intro} (i.e., with $W$ not depending on $k$) we have some technical improvement in comparison with the result of \cite{FarSciSol, SS21}. First of all we do not need anymore to assume the kernels $\rho_k$ to be radially symmetric. Secondly our $\Gamma$-convergence argument is carried out with respect to the convergence in measure instead of the $L^1$ convergence. This is (almost) the natural one for sequences with equibounded energy (see Theorem \ref{thm:main-theorem}-$(ii)$). It can be indeed shown that such sequences are compact in the measure convergence up to an exceptional set $U^\infty$, where their modulus diverges. However, this set can be easily made empty by adding a penalisation term in the energy (see the statement of Theorem \ref{thm:GSBD-comp} and Remark \ref{rem:additional-term}).  \\

Eventually we complement our analysis with a \textit{stochastic homogenisation} result Theorem \ref{thm:stoch-hom}.  Namely we consider  functionals of type \eqref{3intro} with \textit{stationary} random integrands
\begin{equation}\label{5intro}
W_k(\omega,y,M)=W\Big(\omega,\frac{y}{\delta_k},M\Big)\,,
\end{equation}
where $\omega$ belongs to the sample space $\O$ of a probability space $(\O,\T,P)$ and $\delta_k\searrow0$.  Following the approach proposed by \cite{DMM} (which relies on the Subadditive Ergodic Theorem in \cite{AK81}) we show that, almost surely, such functionals $\Gamma$-converge to a free-discontinuity functional of the form \eqref{4intro} where the bulk energy density is independent of the space variable.
A similar result was obtained in \cite{BMZ} in the context of elliptic approximation of free-discontinuity functionals.
\\

\noindent
\textbf{Plan of the paper.} The paper is structured as follows. After fixing the notation, in Section \ref{sec:setting}, we introduce the problem, discuss the assumptions and  state our main results. Section \ref{sec:prelim-results} is devoted to recalling preliminary results which are useful for the analysis. The proof of Theorem \ref{thm:main-theorem} is carried out trough the Sections \ref{sec:compactness}--\ref{sec:ub}, dealing with compactness, lower, and upper bound, respectively. In Section \ref{sec:stoch-hom} we prove a stochastic homogenisation result Theorem \ref{thm:stoch-hom}. Eventually in the Appendix we briefly comment on the result of \cite[Theorem 3.2]{Cortesani}, highlighting that the assumptions made there can actually be weakened. A complete statement is given for the readers convenience in Theorem \ref{cortesani-improved}.

\section{Setting of the problem and main results}\label{sec:setting}
\subsection{Notation}\label{subs:notation} We start by collecting the notation adopted throughout the paper. 

\begin{enumerate}[label=(\alph*)]
	\item $n\geq 2$ is a fixed integer and $p>1$ is a fixed real number; 
	\item  $\mathbb{M}^{n\times n}$  denotes the space of $n\times n$ real matrices; $\mathbb{M}^{n\times n}_{\rm sym}$ and $\mathbb{M}^{n\times n}_{\rm skew}$ denote the spaces of symmetric and skew-symmetric matrices respectively; 
	\item for a subset $A\subset\R^n$ $\partial^*A$ denotes the essential boundary of $A$;
	\item $\mathcal L^n$ and and $\mathcal H^{n-1}$ denote the Lebesgue measure and the $(n-1)$-dimensional Hausdorff measure on $\R^n$, respectively;
	\item for every $A\subset\R^n$ let $\chi_A$ denote the characteristic function of the set $A$;
	\item $U$ denotes an open  bounded subset of $\R^n$ with Lipschitz boundary; 
	\item we denote by $\A(U)$  and $\A$ the collection of all open and bounded subsets  of $U$ and $\R^n$ respectively;
	\item
	If $A,B \in \A(U)$ (or $\A$) by $A \subset \subset B$ we mean that $A$ is relatively compact in $B$;
	\item $Q$ and $Q'$
	denote the open unit cube in $\R^n$ and $\R^{n-1}$ respectively with sides parallel to the coordinate axis, centred at the origin; for $x\in \R^n$ (respectively $x\in \R^{n-1}$) and $r>0$ we set $Q_r(x):= rQ+x$ (respectively $Q'_r(x):= rQ'+x$);
	\item\label{Rn} for every $\xi\in \Sph^{n-1}$ let $R_\xi$ denote an orthogonal $(n\x n)$-matrix such that $R_\xi e_n=\xi$; 
	\item for $x\in\R^n$, $r>0$, and $\xi\in\S^{n-1}$, we define $Q^\xi_r(x):=R_\xi Q_r(x)$. 
	\item for a given topological space $X$, $\mathcal{B}(X)$ denotes the Borel $\sigma$- algebra on $X$. If $X=\R^d$, with $d\in \N$, $d\ge1$ we simply write $\mathcal{B}^d$ in place of $\mathcal B(\R^d)$. For $d=1$ we write $\mathcal B$;
	\item we denote by $L^0(U;\R^n)$ the space of measurable functions;
	\item for $a,b\in \R^n $ the symbol $a
	\otimes b$ denotes the tensor product between $a$ and $b$, while $a\odot b:=\frac12 (a\otimes b+b\otimes a)$. 
\end{enumerate}
\medskip

Throughout the paper $C$ denotes a strictly positive constant which may vary from line to line and within the same expression. 

\subsection{(G)SBV and (G)SBD functions}
We will work with the functional spaces $(G)SBV^p(U;\R^n)$ and $(G)SBD^p(U)$ for which we will recall the main properties and refer the reader to \cite{AFP,DM2013} for a complete exposition of the subject. We say that $u\in L^1(U;\R^n)$ belongs to the space of \textit{special functions with bounded variation}, i.e.,  $u\in SBV(U;\R^n)$, if its distributional gradient is a finite $\mathbb M^{n\times n}$-valued Radon measure without Cantor part, that is,
\begin{equation*}
	Du=\nabla u\L^n+[u]\otimes \nu_u\H^{n-1}\res J_u\,,
\end{equation*}
where $\nabla u$ is the approximate gradient, $J_u$ is the approximate jump set, $[u]=u^+-u^-$ the jump opening and $\nu_u$ the unit normal to $J_u$.   A function $u\in L^0(U;\R^n)$ belongs to the space of  \textit{generalised special functions with bounded variation}, i.e., $u\in GSBV(U;\R^n)$,  if for any $\varphi\in C^1(\R^n;\R^n)$ with support of $\nabla\varphi$ compact it holds $\varphi\circ u\in SBV_{\rm loc}(U;\R^n)$. 
\\

We say that $u\in L^1(U;\R^n)$ belongs to the space of \textit{special functions with bounded deformation}, and we write $u\in SBD(U)$, if its symmetrized distributional gradient is a finite $\mathbb M^{n\times n}_{\rm sym}$-valued Radon measure without Cantor part, that is,
\begin{equation*}
	Eu=\frac{Du+(Du)^T}{2}= e(u)\L^n+ [u]\odot \nu_u\H^{n-1}\res J_u\,,
\end{equation*}
with $e(u)$ is the approximate symmetric gradient with respect to the Lebesgue measure.
On the contrary the space of \textit{generalised special functions with bounded deformation}, $GSBD(U)$, cannot be defined analogously to the space $GSBV(U;\R^n)$ as if $u\in SBD(U)$ and $\varphi$ is as above, then  in general $\varphi\circ u\notin SBD(U)$.  To overcome this issue, Dal Maso in \cite{DM2013} proposed a definition of this space by relying on a slicing argument which we describe in the following.\\
For $\xi\in\R^n\setminus\{0\}$ we let $\Pi^\xi\defas\{y\in\R^n\colon\langle\xi,y\rangle=0\}$; for any $y\in \Pi^\xi$ and $A\in\mathcal{B}(U)$ we set
\begin{equation*}
	A_{\xi,y}\defas\{t\in \R\colon y+t\xi\in A\}\,.
\end{equation*}
Given $u\colon U\to\R^n$ we define $u^{\xi,y}\colon	U_{\xi,y}\to\R$ by 
\begin{equation*}
	u^{\xi,y}(t)\defas\langle u(y+t\xi),\xi \rangle\,.
\end{equation*}
If $u^{\xi,y}\in SBV(U_{\xi,y};\R)$  we set
\begin{equation*}
J_{u^{\xi,y}}^1\defas\{t \in J_{u^{\xi,y}}\colon |[J_{u^{\xi,y}}](t)|\ge1\}\,.
\end{equation*}
We then say that $u\in L^0(U;\R^n)$ belongs to the space of \textit{generalised special functions with bounded deformation}, and we write $u\in GSBD(U)$, if there exists a bounded Radon measure $\lambda$ on $U$ such that $u^{\xi,y}\in SBV_{\rm loc}(U_{\xi,y})$ for all $\nu\in \S^{n-1}$ and all $y\in \Pi^\xi$ and 
\begin{equation*}
	\int_{\Pi^\xi}\Big(|Du^{\xi,y}|(A_{\xi,y}\setminus J_{u^{\xi,y}}^1)+\H^0(A_{\xi,y}\cap J_{u^{\xi,y}}^1)	\Big) \dHn(t)
\le \lambda(A)\,,
\end{equation*}
for all $A\in \mathcal{B}(U)$. Eventually we set
\begin{equation*}
	GSBV^p(U)\defas\{u\in (G)SBV(U;\R^n)\colon \nabla u\in L^p(U;\mathbb M^{n\times n})\ \text{and}\ \H^{n-1}(J_u)<+\infty
	\}\,;\end{equation*}
and
\begin{equation*}
GSBD^p(U)\defas\{u\in (G)SBD(U)\colon e(u)\in L^p(U;\mathbb M^{n\times n}_{\rm sym})\ \text{and}\ \H^{n-1}(J_u)<+\infty
\}\,,
\end{equation*}
where $\nabla u$ and $e(u)$ are well defined also in $GSBV(U;\R^n)$ and $GSBD(U)$ respectively.
%
%
\subsection{Setting of the problem}\label{subs:setting}
Let $1<p<+\infty$; let $c_1,c_2$ be given positive constants such that $0<c_1\le c_2<+\infty$. 
Let $\mathcal{W}\defas\mathcal{W}(p,c_1,c_2)$ be the collection of all functions 
 $ W\colon\R^n\times \mathbb{M}^{n\times n}\to\R$ satisfying the following conditions:
\begin{enumerate}[label=($W\arabic*$)]
	\item\label{hyp:conv+lsc-W} $W$ is a Carathéodory function on $\R^n\times \mathbb{M}^{n\times n}$;
	\item\label{hyp:min-W} $W(x,0)=0$ for every $x\in \R^n$;
		\item\label{hyp:skew} for every $x\in\R^n$, $M\in\mathbb M^{n\times n}$ and $S\in \mathbb M^{n\times n}_{\rm skew}$
	\begin{equation*}
		W(x,M+S)=W(x,M)\,;
	\end{equation*}
	\item\label{hyp:growth-W} for every $x\in \R^n$ and every $M\in \mathbb{M}^{n\times n}$
	\begin{equation*}
		c_1|M+M^T|^p\le W(x,M)\le c_2(|M+M^T|^p+1)\,.
	\end{equation*}
\end{enumerate}
Let $f\colon[0,+\infty)\to[0,+\infty)$ be a lower semi-continuous increasing function such that there exist $\alpha,\beta>0$ with
\begin{equation}\label{hyp:f}
\lim_{t\to 0^+}\frac{f(t)}{t}=\alpha\,,\quad\lim_{t\to+\infty}f(t)=\beta\,.
\end{equation}
Note that for such $f$ it holds 
\begin{equation}\label{eq:upbound-f}
f(t)\le \hat\alpha t\quad \forall \hat{\alpha}>\alpha\,;
\end{equation}
moreover by \cite[Lemma 2.10]{SS21} there exist $(\alpha_i)_{i\in\N}$, $(\beta_i)_{i\in\N}$ sequences of positive numbers with $\sup_i\alpha_i=\alpha$, $\sup_i\beta_i=\beta$ such that
\begin{equation}\label{eq:est-f}
	f(t)\ge f_i(t)\defas\alpha_it\wedge\beta_i\quad\forall i\in\N\,,\ t\in\R\,.
\end{equation}
Let  $\rho\in L^\infty(\R^n;[0,+\infty))$ be a lower semi-continuous convolution kernel with $\int_{\R^n}\rho\dx=1$ and $S\defas{\rm supp} (\rho)$ bounded, convex, symmetrical and that $0\in S$. 
 We denote by $|\cdot|_S$ the norm induced by $S$, namely,
 $$|x|_S\defas\inf\{\lambda>0\colon x\in\lambda S\}\,,$$
 so that  $S=\{|x|_S<1\}$.  Then for any bounded set $K\subset\R^n$ and $x\in\R^n$ we let
 $$\d_S(x,K)\defas\inf_{y\in K}|x-y|_S\,.$$ 
 For any Borel set $E$ and any $r>0$ we denote by $E_r$ and $E_{-r}$ respectively the sets 
 \begin{equation*}
 	E_r\defas\{x\in\R^n\colon\d_S(x,E)<r\}\,,\quad 	E_{-r}\defas\{x\in\R^n\colon\d_S(x,E^c)>r\}\,.
 \end{equation*}
Finally we let $\phi_\rho\colon\R^n\to[0,+\infty)$ be given by
\begin{equation}\label{def:phi}
\phi_\rho(\nu)\defas2\sup_{y\in S}|y\cdot \nu|\,.
\end{equation}
For $\delta>0$ we set $\rho_\delta(x)\defas\frac1{\delta^n}\rho(\frac x{\delta})$, $S_\delta(x)\defas x+\delta S$.\\

For $k\in\N$ let $(W_k)\subset\mathcal W$, let $(\e_k)$ be a decreasing sequence of strictly positive real numbers converging to zero, as $k\to+\infty$ and let $\rho_k\defas\rho_{\e_k}$. 
We consider the family of functionals $F_k\colon  L^0(U;\R^n)\to [0,+\infty]$ defined as 
\begin{equation}\label{F}
	F_k(u)\defas\begin{cases}
		\displaystyle\frac1{\e_k}\int_U f\Big(\e_kW_k(\cdot,e(u))*\rho_k(x)
	\Big)\dx& \text{ if } u\in W^{1,p}(U;\R^n)\,,\\
		+\infty&\text{ otherwise }.
	\end{cases}
\end{equation}
%
%
Let $x\in\R^n$, $M\in \mathbb M^{n\times n}$, $A\in\A$ and $u\in W^{1,p}(A;\R^n)$ be fixed. Set  $u_M(y)\defas My$. We then define the  minimisation problem
\begin{equation}\label{min-prob}
\textbf{m}_k(u_M,A)\defas\inf\left\{\int_AW_k(x,e(v))\dx
\colon v\in W^{1,p}(A;\R^n),\ v=u_M\ \text{near}\ \partial A  \right\}\,,
\end{equation}
  and the cell formulas
\begin{equation}\label{eq:cell-formula-inf}
	\begin{split}
		W'(x,M)\defas \limsup_{r\searrow0^+}\liminf_{k\to+\infty}\frac{\textbf m_k(u_M,Q_r(x))}{r^n}\,,
	\end{split}
\end{equation}
\begin{equation}\label{eq:cell-formula-sup}
	\begin{split}
	W''(x,M)\defas \limsup_{r\searrow0^+}\limsup_{k\to+\infty}\frac{\textbf m_k(u_M,Q_r(x))}{r^n}\,.
	\end{split}
\end{equation}
\subsection{Main Results}\label{sec:main-results}
In this Section we state our main results. The first one is a $\Gamma$-convergence theorem for the energies $F_k$. 
\begin{theorem}[$\Gamma$-convergence of $F_k$]\label{thm:main-theorem}
Let $F_k$ be as in \eqref{F}. Then the following hold:
\begin{enumerate}[label*=$(\roman*)$]
\item  There exists a subsequence, not relabelled, such that $F_k$ $\Gamma$-converges with respect to the convergence in measure to the  functional $F\colon  L^0(U;\R^n) \to [0,+\infty]$ given by
\begin{equation}\label{def:limit-F}
	F(u)\defas\begin{cases}
		\displaystyle\	\alpha\int_UW(x,e(u))\dx+\beta \int_{J_u}\phi_\rho(\nu_u)\dHn&\text{if }u\in GSBD^p(U)\,,\\
		+\infty&\text{otherwise}\,,
	\end{cases}
\end{equation}
with $\phi_\rho$ as in \eqref{def:phi} and for every $x\in U$ and every $M\in\mathbb M^{n\times n}$
\begin{equation*}
	\begin{split}
		W(x,M)=W(x,{\rm sym}(M) )=W'(x,M)=W''(x,M)\,,
	\end{split}
\end{equation*}
with $W''$, $W''$ as in \eqref{eq:cell-formula-inf} and \eqref{eq:cell-formula-sup};
\item Let $(u_k)\subset L^0(U;\R^n)$ be such that $\sup_kF_k(u_k)<+\infty$. Then there exists $u\in GSBD^p(U)$  such that, up to subsequence, it holds
\begin{equation*}
	u_k\to u\quad \text{in measure on }U\setminus U^\infty\,,
\end{equation*}
\begin{equation*}
	e(u_k)\rightharpoonup e(u)\quad \text{in } L^p(U\setminus U^\infty;\mathbb{M}^{n\times n}_{\rm sym})\,,
\end{equation*}
\begin{equation*}
\liminf_{k\to+\infty}\H^{n-1}(J_{u_k})\ge \H^{n-1}(J_u\cup(\partial^*U^\infty\cap U))\,,
\end{equation*}
where $U^\infty\defas\{x\in U\colon |u_k(x)|\to+\infty\}$.
If in addition 
\begin{equation*}
	\sup_{k\in\N}\int_U\psi(|u_k|)\dx<+\infty\,,
\end{equation*}
for some $\psi\colon[0,+\infty)\to [0,+\infty)$,  continuous, increasing with $\lim_{s\to+\infty}\psi(s)=+\infty$, then $U^\infty=\emptyset$, so that $|u|$ is finite a.e., and $u_k\to u$ in measure on $U$.  
\end{enumerate}

\end{theorem}
\begin{remark}\label{rem:additional-term}
The addition of a penalty term of the form 
$\int_U\psi(|u|)\dx$ to the energy enforces then compactness in measure, while 
causing no troubles in the $\Gamma$-convergence analysis. Indeed, such a term is clearly lower semicontinuous, hence the corresponding lower bound follows immediately. As for the upper bound, if one takes $\psi$ as in Theorem \ref{thm:GSBD-density}, the argument of Proposition \ref{prop:upper-bound} can be readily adapted also in presence of such an additional term. As this is not the core of the argument, we will neglect lower order terms in our statements and proofs, directly assuming that convergence in measure holds everywhere. The technical details left to prove the upper bound are summarized in Remark \ref{rem:ub} for the readers convenience.
\end{remark}
The proof of Theorem \ref{thm:main-theorem} is divided into three main steps contained respectively in sections \ref{sec:compactness}, \ref{sec:lb} and \ref{sec:ub}.
As a consequence of Theorem \ref{thm:main-theorem} and the Urysohn property of $\Gamma$-convergence \cite[Proposition 8.3]{DM} we deduce the following corollary.
\begin{corollary}\label{cor:urysohn}
	Let $(W_k)\subset\mathcal{W}$ and let $F_k$ be the functionals as in \eqref{F}. Let $W'$, $W''$ be as in \eqref{eq:cell-formula-inf} and \eqref{eq:cell-formula-sup}, respectively. Assume that 
	\begin{equation*}
	W'(x,M)=W''(x,M)=:W(x,M)\,, \quad \text{for a.e. } x\in\R^n\text{ and for every }M\in\mathbb M^{n\times n}\,,
	\end{equation*}
for some Borel function $W\colon\R^n\times\mathbb M^{n\times n}\to[0,+\infty)$. Let  $F$ defined as in \eqref{def:limit-F} accordingly.
Then the functionals $F_k$ $\Gamma$-converge with respect to the convergence in measure to $F$.
Moreover
\begin{equation*}
	\begin{split}
		W(x,M)=W(x,{\rm sym}(M) )=W'(x,M)=W''(x,M)\,,
	\end{split}
\end{equation*}
for every $x\in U$ and every $M\in\mathbb M^{n\times n}$.

\end{corollary}
We now state a homogenisation theorem without assuming any spatial periodicity of the energy densities $W_k$. We start by introducing some notation. We fix $W\in\mathcal W$ and set
\begin{equation}\label{min-prob-hom}
	\m(u_M,A)\defas\inf\left\{\int_AW(x,e(v))\dx\colon v\in W^{1,p}(A;\R^n),\ v=u\ \text{near}\ \partial A
	\right\}\,,
\end{equation}
for all $A\in\mathcal{A}$ and all $M\in\mathbb M^{n\times n}$
Let also $(W_k)\subset\mathcal W$ be given by
\begin{equation}\label{def:W-eps-hom}
	W_k(x,M)\defas W\Big(\frac{x}{\delta_k},M\Big)\,,
\end{equation}
with $\delta_k\searrow0$ when $k\to+\infty$. 
\begin{theorem}[Deterministic homogenisation]\label{thm:det-hom} Let 
Let $W\in\mathcal W$ and let $\m(u_M,Q_t(tx))$ be as in \eqref{min-prob-hom} with $A=Q_t(tx)$. Assume that for every $x\in\R^n$, $M\in\mathbb M^{n\times n}$ the following limit
\begin{equation}\label{eq:lim-cell-hom}
\lim_{t\to+\infty}\frac{\m(u_M,Q_t(tx))}{t^n}=:W_{\rm hom}(M)\,,
\end{equation}
exists and is independent of $x$. Then the functionals $F_k$ defined in \eqref{F} with $W_k$ as in \eqref{def:W-eps-hom} $\Gamma$-converge with respect to the convergence in measure to the functional  $F_{\rm hom}\colon L^0(U;\R^n)\to[0,+\infty]$ given by 
\begin{equation}\label{def:limit-F-hom}
	F_{\rm hom}(u)\defas\begin{cases}
		\displaystyle\	\alpha\int_U W_{\rm hom}(e(u))\dx+\beta \int_{J_u}\phi_\rho(\nu_u)\dHn&\text{if }u\in GSBD^p(U)\,,\\
		+\infty&\text{otherwise}\,,
	\end{cases}
\end{equation}
with $\phi_\rho$ as in \eqref{def:phi}. Moreover $W_{\rm hom}(M)=W_{\rm hom}({\rm sym}(M))$ for all $M\in\mathbb M^{n\times n}$.
\end{theorem}
\begin{proof}
	Let $W'$, $W''$ be respectively as in \eqref{eq:cell-formula-inf} and \eqref{eq:cell-formula-sup}. By Corollary \ref{cor:urysohn} it is sufficient to show that
	\begin{equation}\label{eq:lim-hom}
W_{\rm hom}(M)=W'(x,M)=W''(x,M)\,,
	\end{equation}
for all $x\in \R^n$ and $M\in\mathbb M^{n\times n}$. We fix $x\in\R^n$, $M\in \mathbb M^{n\times n}$, $r>0$ and $k\in\N$. For any $u\in W^{1,p}(Q_r(x);\R^n)$ with $u=u_M$ near $\partial Q_r(x)$
we let $u_k\in W^{1,p}(Q_{\frac r{\delta_k}}(\frac x{\delta_k});\R^n)$ be given by $u_k(y)\defas \frac1{\delta_k}u(\delta_ky)$. Then clearly $u_k=u_M$ near $\partial Q_{\frac r{\delta_k}}(\frac x{\delta_k})$. Moreover by performing the change of variable $\hat y=\frac y{\delta_k}$ we find
\begin{equation*}
	\int_{Q_r(x)}W\Big(\frac y{\delta_k},e(u)\Big)\dy=\delta_k^n \int_{Q_{\frac r{\delta_k}}(\frac x{\delta_k})} W(y,e(u_k))\dy\,.
\end{equation*}
Hence in particular 
\begin{equation*}
\m_k(u_M,Q_r(x))=\delta_k^n\m(u_M,Q_{\frac r{\delta_k}}(\tfrac{x}{\delta_k}))=\frac{r^n}{t_k^n}
\m(u_M,Q_{t_k}(t_k\tfrac{x}{r}))\,,
\end{equation*}
with $t_k\defas \frac r{\delta_k}$.
Eventually passing to the limit as $k\to+\infty$ by \eqref{eq:lim-cell-hom} we deduce
\begin{equation*}
\lim_{k\to+\infty} \frac{\m_k(u_M,Q_r(x))}{r^n}=
\lim_{k\to+\infty} \frac{\m(u_M,Q_{t_k}(t_k\tfrac{x}{r}))}{t_k^n}=W_{\rm hom}(M)\,.
\end{equation*}
\end{proof}
\section{Some preliminary results} \label{sec:prelim-results}
In this section we collect some useful results that will be employed throughout the paper.
We start by recalling a $\Gamma$-convergence result for the bulk energies defined in \eqref{E} (Theorem \ref{thm:conv-in-sobolev}) and a $\Gamma$-convergence result for one-dimensional non-local energies (Theorem \ref{thm:1dim}). To follow we recall a density and a compactness result (cf., Theorem \ref{thm:GSBD-density} and Theorem \ref{thm:GSBD-comp}). We conclude this section with a series of technical lemmas (cf. Lemmas \ref{technical-lemma1}, \ref{technical-lemma2}, \ref{lem:convL1} and Corollary \ref{cor:convL0}).\\

We consider the family of functionals $E_k\colon  L^0(\R^n;\R^n)\times\mathcal A\to [0,+\infty]$ given by
\begin{equation}\label{E}
	E_k(u,A)\defas\begin{cases}
		\displaystyle\int_A W_k(x,e(u))
		\dx& \text{ if } u\in W^{1,p}(A;\R^n)\,,\\
		+\infty&\text{ otherwise }.
	\end{cases}
\end{equation}
\begin{theorem}[$\Gamma$-convergence of $E_k$]\label{thm:conv-in-sobolev}
	Let $E_k$ be as in \eqref{E}.
Then there exists a subsequence, not relabelled, such that for every $A\in\mathcal{A}$ the functionals $E_k(\cdot,A)$ $\Gamma$-converge, with respect to the convergence in measure, to the  functional $E(\cdot ,A)$ with
 $E\colon  L^0(\R^n;\R^n)\times \mathcal{A}\to [0,+\infty]$ given by
	\begin{equation}\label{def:E(u)}
		E(u,A)=\begin{cases}
			\displaystyle\int_A W(x,e(u))
			\dx& \text{ if } u\in W^{1,p}(A;\R^n)\,,\\
			+\infty&\text{ otherwise }\,,
		\end{cases}
	\end{equation} 
	where for every $x\in\R^n$ and every $M\in\mathbb M^{n\times n}$
	\begin{equation}\label{def:W}
		\begin{split}
			W(x,M)=W(x,{\rm sym}(M) )=W'(x,M)=W''(x,M)\,,
		\end{split}
	\end{equation}
with $W'$, $W''$ as in \eqref{eq:cell-formula-inf} and \eqref{eq:cell-formula-sup} respectively. The same $\Gamma$-convergence holds with respect to the $L^p_{\rm loc}(\R^n;\R^n)$ convergence.
\end{theorem}
The proof of Theorem \eqref{thm:conv-in-sobolev} is rather standard and follows by  the localisation method (see e.g., \cite[Sections 18,19]{DM}) and by suitably adapting the integral representation result in \cite[Theorem 2]{BFLM} to our setting with the help of Korn-Poincaré inequality. For this reason we omit the proof here and we refer the reader to \cite[Proposition 3.13]{FPS} for more details. We only highlight that the result holds also for non regular open bounded subsets of $\R^n$. Since this may not be immediately clear from the statement given in \cite[Proposition 3.13]{FPS}, we discuss this point in the remark below.
\begin{remark}\label{rem:conv-in-sobolev}
Let $A$ be \textit{any} open bounded subset of $\R^n$ and $u\in W^{1,p}(A ;\R^n)$. We show that there exists a sequence $(u_k)\subset W^{1,p}(A;\R^n)$ such that $u_k\to u$ in $L^p(A;\R^n)$ and $E_k(u_k,A)\to E(u,A)$. With the use of Korn-Poincaré inequality, it is clear that this can be done if $A$ is an extension domain. In the general case, consider smooth relatively compact subsets $A'\subset\subset A''\subset \subset A$, and fix $\eta >0$. We find a sequence $(v_k)\subset W^{1,p}(A'';\R^n)$ such that $E_k(v_k,A'')\to E(u,A'')$. By the liminf inequality, this also gives $E_k(v_k,A''\setminus A')\to E(u,A''\setminus A')$. With \ref{hyp:growth-W} we have
\[
\limsup_{k\to +\infty}\int_{A''\setminus A'}|e(v_k)|^p\dx\le \frac{c_2}{c_1}\int_{A''\setminus A'}(1+|e(u)|^p)\dx
\] 
Then, considering a cut-off $\varphi$  between $A'$ and $A''$ we set $u_k\defas\varphi v_k+(1-\varphi)u$. Clearly $u_k\to u$ in $L^p(A;\R^n)$. Furthermore, by \ref{hyp:growth-W} one has
\begin{equation*}
	\begin{split}
\limsup_{k\to +\infty}E_k(u_k,A)&=\limsup_{k\to +\infty}E_k(v_k,A')+E_k(u_k,A\setminus A')\\
&
\le \limsup_{k\to +\infty}E_k(v_k,A'')+ c_2\left[\int_{A\setminus A'}(1+|e(u)|^p)\dx\right.\\
&\left.+\int_{A''\setminus A'}(1+|\nabla \phi \odot(v_k-u)|^p+|e(v_k)|^p)\dx\right]
\\
&\le E(u,A'')+c_2\left(1+\frac{c_2}{c_1}\right)\int_{A\setminus A'}(1+|e(u)|^p)\dx \le E(u,A)+\eta\,,
	\end{split}
\end{equation*}
provided $\L^n(A\setminus A')$ is sufficiently small. The limsup inequality, which is the only relevant one, follows by a diagonal argument. 
\end{remark}
We recall now the following one-dimensional result for non-local energies given in \cite[Theorem 3.30]{Br}.
\begin{theorem}[$\Gamma$-convergence in 1d]\label{thm:1dim}
	Let $I\subset \R$ be a bounded interval. Let $f\colon[0,+\infty)\mapsto [0,+\infty)$ be a lower semi-continuos function satisfying \eqref{hyp:f} for some $\alpha,\beta>0$. 
	 Consider the family of functionals $G_k\colon L^0(I)\to [0,+\infty]$  defined by 
	\begin{equation*}
		G_k(w):=\frac1{\e_k}\int_If\Big(\frac12\int_{x-\e_k}^{x+\e_k}|\dot w(y)|^p\dy\Big)\dx\,,
		\end{equation*}
	if $u\in W^{1,p}(I)$ and $+\infty$ otherwise. Then $G_k$ $\Gamma$-converge with respect to the convergence in measure, to the functional $G\colon L^0(I)\to[0,+\infty]$ given by
	\begin{equation*}
		G(w):=\alpha\int_I|\dot w|^p\dx+ 2\beta\#(J_w)\,,
	\end{equation*}
if $w\in SBV(I)$ and $+\infty$ otherwise.
\end{theorem}
We next recall an approximation result \cite[Theorem 1.1]{Chamb-Cris2019} and a compactness result  in $GSBD^p$ in \cite{Chamb-Cris2021} (which generalises \cite[Theorem 11.3]{DM2013}).  To this aim we denote by $\mathcal W^\infty_{\rm pw}(U;\R^{n})\subset GSBD^p(U)$ the space of ``\textit{piecewise smooth}'' $SBV$-functions, that is,
\begin{equation}\label{def:W_pc}
	\begin{split}
\mathcal	W^\infty_{\rm pw}(U;\R^{n})\defas \Big\{
u\in GSBD^p(U)\colon& u\in SBV(U;\R^n)\cap W^{m,\infty}(U\setminus J_u;\R^n), \ \forall m\in \N, \\
& \H^{n-1}(\overline J_u\setminus J_u)=0,\ \overline{J}_u=\cup_{i=1}^kK_i\subset\subset U \\
&\text{with $K_i$ connected (n-1)-rectifiable set}, \forall\, 1\le i\le k
\Big\}
	\end{split}
\end{equation}
\begin{theorem}[Density in $GSBD^p$]\label{thm:GSBD-density}
	Let $\phi$ be a norm on $\R^n$. Let $u\in GSBD^p(U)$. Then there exists a sequence $(u_j)\subset\mathcal W^\infty_{\rm pw}(U;\R^n)$ such that 
	\begin{enumerate}[label=$(\roman*)$]
		\item $u_j\to u\ \text{in measure on }U$\,;
		\item $e(u_j)\to e(u)\ \text{in } L^p(U;\mathbb M^{n\times n}_{\rm sym})$\,;
		\item $\lim_{j\to\infty} \int_{J_{u_j}}\phi(\nu_{u_j})\dHn= \int_{J_{u}}\phi(\nu_{u})\dHn$\,.
	\end{enumerate}
Moreover, if 
\begin{equation*}
\int_U\psi(|u|)\dx<+\infty\,,
\end{equation*}
for some $\psi\colon[0,+\infty)\to [0,+\infty)$,  continuous, increasing with
\begin{equation*}
	\psi(0)=0\,,\ \psi(s+t)\le C(\psi(s)+\psi(t))\,,\ \psi(s)\le C(1+s^p)\,, \lim_{s\to+\infty}\psi(s)=+\infty\,;
\end{equation*}
then 
\begin{equation*}
	\lim_{j\to+\infty}\int_U\psi(|u_j-u|)\dx=0\,.
\end{equation*}
\end{theorem}
We notice that the approximating class considered above  fulfils the additional requirement of having a jump set compactly contained in $U$. This is possible, as shown in \cite[Theorem C]{DePhFuPra}.
\begin{theorem}[Compactness in $GSBD^p$]\label{thm:GSBD-comp}
	Let $(u_j)\subset GSBD^p(U)$ be a sequence satisfying
	\begin{equation*}
\sup_{j\in\N}\Big(\|e(u_j)\|_{L^p(U)}+\H^{n-1}(J_{u_j})
\Big)<+\infty\,.
	\end{equation*}
Then there exist a subsequence, still denoted by $(u_j)$, and $u\in GSBD^p(U)$ with the following properties:
\begin{enumerate}[label=$(\roman*)$]
\item the set $U^\infty\defas\{x\in U\colon|u_j|\to+\infty\}$ has finite perimeter;
\item $u_j\to u$ in measure on $U\setminus U^\infty$\,;
\item $e(u_j)\rightharpoonup e(u)$ in $L^p(U\setminus U^\infty;\mathbb{M}^{n\times n}_{\rm sym})$\,;
\item $\liminf_{j\to+\infty}\H^{n-1}(J_{u_j})\ge \H^{n-1}(J_u\cup(U\cap\partial^*U^\infty))$\,.
\end{enumerate}
If in  addition
\begin{equation*}
	\sup_{j\in\N}\int_U\psi(|u_j|)\dx<+\infty\,,
\end{equation*}
for some $\psi\colon[0,+\infty)\to [0,+\infty)$,  continuous, increasing with $\lim_{s\to+\infty}\psi(s)=+\infty$, then $U^\infty=\emptyset$ so that $|u|$ is finite a.e. and $(ii)$ holds in $U$.  
\end{theorem}
In the rest of this section we give some technical Lemmas.
\begin{lemma}\label{technical-lemma1}
For $j\in \N$  let $g_j\colon \R^n\to [0,+\infty)$ be a family of equi-integrable functions. Let $E_j\subset\R^n$ be such that
$\L^n(E_j)\to0$ and let $\delta_j\searrow0$ as $j\to+\infty$.  Then $(g_j\chi_{E_j})*\rho_{\delta_j}$ converges to 0 strongly in $L^1(\R^n)$.
\end{lemma}
\begin{proof}
	By properties of convolution it holds that
	\begin{equation*}
	\|g_j\chi_{E_j}*\rho_{\delta_j}\|_{L^1(\R^n)}\le \|g_j\chi_{E_j}\|_{L^1(\R^n)}.
	\end{equation*}
By equi-integrability we have that for every $\eps>0$ there is $ J\in \N$ such that for every $j\ge J$ 
\begin{equation*}
 \|g_j\chi_{E_j}\|_{L^1(\R^n)}=\int_{E_j}g_j\dx\le \eps,
\end{equation*}
from which the thesis follows.
\end{proof}
\begin{lemma}\label{technical-lemma2}
Let $A'$ be an open bounded subset of $\R^n$.	For $j\in \N$  let $g_j\colon A'\to [0,+\infty)$ be a family of equi-integrable functions.  Let $\delta_j\searrow0$ as $j\to+\infty$. Then for every $A\subset\subset A'$  there holds
	\begin{equation*}
\liminf_{j\to+\infty}\int_{A}g_j*\rho_{\delta_j}\dx\ge
\liminf_{j\to+\infty}\int_{A}g_j\dx\,.
	\end{equation*}
\end{lemma}
\begin{proof}
We consider the sequence of positive measures $\nu_j\defas g_j*\rho_{\delta_j}\L^n\res A$.  Since $A'$ is bounded $g_j$ turn out to be equi-bounded in $L^1(A')$, hence we get
\begin{equation*}
	\nu_j(A)= \int_Ag_j*\rho_{\delta_j}\dx \le  \int_{A'}g_j\dx \le C\,.
\end{equation*}
Therefore there exist a positive measure $\nu\in \mathcal{M}_b(A)$, a function $g\in L^1(A')$, and a not-relabelled subsequence such that $\nu_j\stackrel{*}{\rightharpoonup}\nu$ weakly $*$ in $\mathcal{M}_b(A)$ and $g_j\to g$ weakly in $L^1(A')$.
It remains to show that $\nu=g\L^n\res A$, indeed this would imply
\begin{equation*}
	\liminf_{j\to+\infty}\int_Ag_j*\rho_{\delta_j}\dx=
\liminf_{j\to+\infty}\nu_j(A)\ge \nu(A)=\int_Ag\dx=\liminf_{j\to+\infty}\int_Ag_j\dx\,,
\end{equation*}
and we could conclude. Let $\varphi\in C^\infty_c(A)$  and let $A\subset\subset A''\subset\subset A'$ be fixed. By Fubini's theorem we have
\begin{equation*}\begin{split}
			\int_A\varphi \d\nu =\lim_{j\to+\infty}\int_A\varphi \d\nu_j&=\lim_{j\to+\infty}\int_A\varphi (g_j*\rho_{\delta_j})\dx\\
			&=\lim_{j\to+\infty}\int_{A''}(\varphi*\hat\rho_{\delta_j} )g_j\dx=\int_{A''}\varphi g\dx =\int_{A}\varphi g\dx,
	\end{split}
\end{equation*}
where $\hat\rho_{\delta_j}(x)\defas\rho_{\delta_j}(-x)$ and the last inequality follows since $g_j\rightharpoonup g$ weakly in $L^1(A')$ and $\varphi*\hat\rho_{\delta_j}(x)\to\varphi$ strongly in $L^1(A')$. Thus we deduce $\nu=g\L^n\res A$ and the proof is concluded.
\end{proof}
\begin{lemma}\label{lem:convL1}
	Let  $A\subset \R^{n-1}$.
Let $(u_k)\subset L^1(A)$ be a sequence converging to $u$ in $L^1(A)$. Let $A'\subset\subset A$ and let $w_k\colon A'\times Q'\to\R$ be given by $w_k(x,y)\defas u_k(x+\e_k y)$. Then $w_k$ converges to $u$ in $L^1(A'\times Q')$.
\end{lemma}
\begin{proof}
	By Frechet-Kolmogoroff's Theorem for every $\eta>0$, $y\in Q'$ there is $ h\in\N$ such that for all $k\ge h$ there holds
	\begin{equation*}
\int_{A'}|u_k(x+\e_k y)-u(x)|\dx\le \eta\,.
	\end{equation*}
	This together with Fubini's theorem yield
	\begin{equation*}
		\begin{split}
\int_{A'\times Q'}|w_k(x,y)-u(x)|\dx\dy\le \int_{Q'}\int_{A'}|u_k(x+\e_k y)-u(x)|\dx\dy\le \eta \,,
		\end{split}
	\end{equation*}
for all $k\ge h$. Eventually by letting $\eta\to0$ we conclude.
\end{proof}
\begin{corollary}\label{cor:convL0}
	Let  $A\subset \R^{n-1}$.
Let $(u_k)\subset L^0(A)$ be a sequence converging to $u$ in measure. Let $A'\subset\subset A$ and let $w_k\colon A'\times Q'\to\R$ be given by $w_k(x,y)\defas u_k(x+\e_k y)$. Then $w_k$ converges to $u$ in measure.
\end{corollary}
\begin{proof}
Since  ${\rm arctan}(u_k)$ converges to ${\rm arctan}(u)$ in $L^1(A)$ by Lemma \ref{lem:convL1} we have that  ${\rm arctan}(w_k)$ converges to ${\rm arctan}(u)$ in $L^1(A'\times Q')$. Hence $w_k$ converges to $u$ in measure.
\end{proof}
\section{Compactness}\label{sec:compactness}
In this section we prove  point $(ii)$ of Theorem \ref{thm:main-theorem}.
\begin{proposition}[Compactness]\label{prop:compactness}
 Let $F_k$ be as in \eqref{F}.  Let $(u_k)\subset L^0(U;\R^n)$ be such that $\sup_kF_k(u_k)<+\infty$. Then there exists $u\in GSBD^p(U)$  such that, up to subsequence, it holds
\begin{equation*}
	u_k\to u\quad \text{in measure on }U\setminus U^\infty\,,
\end{equation*}
\begin{equation*}
	e(u_k)\rightharpoonup e(u)\quad \text{in } L^p(U\setminus U^\infty;\mathbb{M}^{n\times n}_{\rm sym})\,,
\end{equation*}
\begin{equation*}
	\liminf_{k\to+\infty}\H^{n-1}(J_{u_k})\ge \H^{n-1}(J_u\cup(\partial^*U^\infty\cap U))\,,
\end{equation*}
where $U^\infty\defas\{x\in U\colon |u_k(x)|\to+\infty\}$. If in addition 
\begin{equation*}
	\sup_{k\in\N}\int_U\psi(|u_k|)\dx<+\infty\,,
\end{equation*}
for some $\psi\colon[0,+\infty)\to [0,+\infty)$,  continuous, increasing with $\lim_{s\to+\infty}\psi(s)=+\infty$, then $U^\infty=\emptyset$, and all implications hold on $U$.  
\end{proposition}
\begin{proof}
The proof follows by suitably adaptations of proof of \cite[Proposition 4.1]{FarSciSol}. We recall it here for completeness.
Let $(u_k)$ be as in the statement and let $U'\subset\subset U$ be fixed. Then it is sufficient to prove the following claim: there exist ($\bar u_k)\subset GSBV^p(U;\R^n)$ and $c_0>0$ (independent of $k$) such that 
\begin{equation}\label{eq:claim1}
\bar u_k-u_k\to0\ \text{in measure on }U\,,
\end{equation}
\begin{equation}\label{eq:claim2}
F_k(u_k)\ge c_0\bigg(\int_{U'} |e(\bar u_k)|^p\dx+ \H^{n-1}(J_{\bar u_k})\bigg)\,.
\end{equation}
Indeed by \eqref{eq:claim2} and Theorem \ref{thm:GSBD-comp} we deduce the existence of $u\in GSBD^p(U')$ such that, up to subsequence, 
\begin{equation*}
\bar u_k\to u\ \text{in measure on } \O'\setminus {\O'}^\infty\,,
\end{equation*}
\begin{equation*}
e(\bar u_k)\rightharpoonup e(u)\ \text{in}\ L^p(U'\setminus {U'}^\infty;\mathbb{M}^{n\times n}_{\rm sym})\,,
\end{equation*}
and 
\begin{equation*}
\liminf_{k\to+\infty}\H^{n-1}(J_{\bar u_k}\cap U')\ge \H^{n-1}(J_u\cup({U'}\cap\partial^*{U'}^\infty))\,.
\end{equation*}
 Eventually by \eqref{eq:claim1} and the arbitrariness of $U'$ we would conclude (observing indeed that also the remaining part of the statement follows directly from Theorem \ref{thm:GSBD-comp}).
Thus we are only left to prove the claim. \\For fixed $i\in\N$ let $f_i(t)=\alpha_i\wedge \beta_i$ be as in \eqref{eq:est-f}. Choose $\eta\in(0,1)$ such that $Q_\eta(0)\subset\subset S$ and let
 \begin{equation}\label{def:m-eta,f-eta}
 	m_\eta\defas\min_{x\in\overline Q_\eta(0)}\rho(x)>0\quad \text{and}\quad
 	f_i^\eta(t)\defas f_i(m_\eta\eta^nt)=\alpha_im_\eta\eta^nt\wedge \beta_i\,.
 \end{equation}
 Then  we have 
 \begin{equation}\label{eq:comp1}
 	\begin{split}
 	F_k(u_k)&\ge \frac1{\e_k}\int_U f_i\Big(\e_kW_k(\cdot,e(u_k))*\rho_k(x)
 	\Big)\dx\\
 	&\ge \frac1{\e_k}	\int_{ U}f_i^\eta\Big(\e_k\strokedint_{Q_{\eta\e_k}(x)}
 	W_{k}(y,e(u_{k}))\dy\Big)\dx
 	\,.
 	\end{split}
 \end{equation} We  set
\begin{equation*}
	A_k^1\defas\Biggl\{x\in U\colon \e_k\strokedint_{Q_{\eta\e_k(x)}}
	W_{k}(y,e(u_{k}))\dy\ge \frac{\beta_i}{\alpha_im_\eta\eta^{2n}}
	\Biggr\}\,,
\end{equation*}
\begin{equation*}
	A_k^2\defas\Biggl\{x\in U\colon \dist(x,A_k^1)\le (1-\eta)\e_{k}
	\Biggr\}\,.
\end{equation*}
Note that 
\begin{equation}\label{containment-1}
	A_k^1\subset	A_k^2\subset\Biggl\{x\in U\colon \e_k\strokedint_{Q_{\e_{k}}(x)}
	W_{k}(y,e(u_{k}))\dy\ge \frac{\beta_i}{\alpha_im_\eta\eta^{n}}
	\Biggr\}\,.
\end{equation}
Indeed if $x\in A_k^2$ there is $z\in 	A_k^1 $ with $Q_{\eta\e_{k}}(z)\subset Q_{\e_{k}}(x)$ and therefore 
\begin{equation*}
\e_k\strokedint_{Q_{\e_k}(x)}
	W_{k}(y,e(u_{k}))\dy\ge \eta^n \e_{k}\strokedint_{Q_{\eta\e_{k}}(z)}
	W_{k}(y,e(u_{k}))\dy\ge \frac{\beta_i}{\alpha_im_\eta\eta^{n}}\,.
\end{equation*} 
By combining together \eqref{eq:comp1} and  \eqref{containment-1} we find
\begin{equation}\label{eq:comp3}
	F_k(u_k)
	\ge\frac{\beta_i}{\e_k}\L^n(A_k^2)
	\,.
\end{equation}
By the coarea formula (see e.g., \cite[Theorem 3.14]{EvGar}) and the mean value theorem there exists $t_k\in(0,(1-\eta)\e_{k})$ such that the set $A_k^3\defas\{\dist(\cdot,A_k^1)\le t_k
\}\subset A_k^2$ satisfies
\begin{equation}\label{eq:comp4}
	\L^n(A_k^2)\ge (1-\eta)\e_{k}\H^{n-1}(\partial A_k^3 )\,.
\end{equation}
Let now
\begin{equation*}
	\bar u_k(x)\defas\begin{cases}
		0&\text{ if } x\in A_k^3\,,\\
		u_{k}&\text{ otherwise in }U\,.
	\end{cases}
\end{equation*}
By \eqref{eq:comp3} and the fact that $A_k^3\subset A_k^2$ we have $\L^n(A_k^3)\to0$ as $k\to+\infty$ from which \eqref{eq:claim1} follows.
On the other hand as $J_{\bar u_k}=\partial A_k^3$ \eqref{eq:comp3} and \eqref{eq:comp4} yield
\begin{equation}\label{eq:comp4bis}
	F_k(u_k)\ge (1-\eta)\beta_i\H^{n-1}(J_{\bar u_k})\,.
\end{equation}
We next show that there exists $K(n)\ge 1$ such that  for every $x\in \overline Q$
\begin{equation}\label{eq:comp5}
c_1\e_k\strokedint_{Q_{\eta\e_{k}}(x)}
|e(\bar u_{k}(y))|^p	\dy\le K \frac{\beta_i}{\alpha_im_\eta\eta^n}\,.
\end{equation}
By \ref{hyp:growth-W} we have
\begin{equation}\label{eq:comp5bis}
W_{k}(x, e(u_{k}(x)))\ge c_1|e(u_{k}(x))|^p\ge c_1|e(\bar u_k(x))|^p\quad \text{for a.e. }x\in U\,.
\end{equation}
Now if $x\in U\setminus A_k^3$,  then $x\notin A_k^1$ and
\begin{equation*}
c_1\e_k\strokedint_{Q_{\eta\e_{k}}(x)}
|e(\bar u_{k}(y))|^p	\dy\le \e_{k}
\strokedint_{Q_{\eta\e_{k}}(x)}W_{k}(y, e(u_{k}(y)))\dy\le \frac{\beta_i}{\alpha_im_\eta\eta^n}\,.
\end{equation*}
Assume instead that $x\in A_k^3$. Observe  that $\bar u_k=0$ in $Q_{\eta\e_{k}}(x)\cap A_k^3$, so that
\[
\int_{Q_{\eta\e_{k}}(x)}
|e(\bar u_{k}(y))|^p	\dy=\int_{Q_{\eta\e_{k}}(x)\setminus A_k^3}
|e(\bar u_{k}(y))|^p	\dy\,.
\]
Furthemore, we can cover $Q_{\eta\e_{k}}(x)\cap (Q\setminus A_k^3)$ with a finite number $ K(n)\ge1$ of balls of radius $\eta\e_{k}$ and centres $x_1,...,x_{ K}\in U\setminus A_k^3$ (see e.g. \cite[Remark 2.8]{SS21}). Hence, we find
\begin{equation*}
\begin{split}
c_1\e_k\strokedint_{Q_{\eta\e_{k}}(x)}
|e(\bar u_{k}(y))|^p	\dy&
\le c_1\e_k\sum_{i=1}^{K}\strokedint_{Q_{\eta\e_{k}}(x_i)}
|e(\bar u_{k}(y))|^p	\dy\\
&\le \e_k\sum_{i=1}^{K}
\strokedint_{Q_{\eta\e_{k}}(x_i)}W_{k}(y, e(u_{k}(y)))\dy
\le  K \frac{\beta_i}{\alpha_im_\eta\eta^n}\,,
\end{split}
\end{equation*}
and \eqref{eq:comp5} follows.
Finally by \eqref{eq:comp1}, the monotonicity of $f_i^\eta$, \eqref{eq:comp5bis}
we infer
\begin{equation}\label{eq:comp6}
\begin{split}
F_k(u_k)	&\ge \frac1{\e_k}	\int_{ U}f_i^\eta\Big(\e_k\strokedint_{Q_{\eta\e_k}(x)}c_1
|e(\bar u_{k}(y))|^p
\dy\Big)\dx\\
&
\ge c_1 \frac{\alpha_im_\eta\eta^n}{K} 	\int_{ U}
\strokedint_{Q_{\eta\e_{k}}(x)}
|e(\bar u_{k}(y))|^p	\dy\dx\,,
\end{split}
\end{equation}
where the last inequality follows from \eqref{eq:comp5} and the fact that $f_i^\eta(t)\ge  \frac{\alpha_im_\eta\eta^n}{K} t$  when $t\le K \frac{\beta_i}{m_\eta\eta^n}$.
Moreover by using in order the  change of variable $y=x-\eta\e_{k} z$, Fubini's theorem, and the change of variable $\hat x=x-\eta\e_kz$ (for $k$ large enough), we find
\begin{equation}\label{eq:comp7}
\begin{split}
\int_{ U}
\strokedint_{Q_{\eta\e_{k}}(x)}
|e(\bar u_{k}(y))|^p	\dy\dx&=\int_{ Q}\int_{ U}|e(\bar u_{k}(x-\eta\e_{k} z))|^p	\dx\dz\\
& \ge \int_{U'} |e(\bar u_k(x))|^p\dx\,.
\end{split}
\end{equation}
Eventually gathering together \eqref{eq:comp4}, \eqref{eq:comp6}, and \eqref{eq:comp7},  we deduce \eqref{eq:claim2} with $$c_0\defas \frac12\Big(\frac{c_1 \alpha_im_\eta\eta^n}{K}\wedge(1-\eta)\beta_i\Big)\,,$$
and in particular, by arbitrariness of $U'$, $\bar u_k\in GSBV^p(U;\R^n)$. 
\end{proof}
%
\section{Lower bound}\label{sec:lb}
In this section we prove the lower bound. To this purpose it is convenient to localise the functionals $F_k$, namely we set
\begin{equation}\label{def:localised-F_k}
	F_k(u,A)\defas \frac1{\e_k}\int_A f\Big(\e_kW_k(\cdot,e(u))*\rho_k(x)
	\Big)\dx \,,\quad \text{for } u\in W^{1,p}(U)\,,\ A\subset U\,.
\end{equation}
\begin{proposition}[Lower bound: bulk contribution]
	\label{blow-up} 
	  Let $(u_k)\subset L^0(U;\R^n)$ be a sequence  that converges to  in measure to $u\in  L^0(U;\R^n)$. Assume moreover  that $F_k(u_k)\le C$ .  Then there exists a subsequence, not relabelled, 
	  such that 
	\begin{equation*}
		\liminf_{k\to+\infty}F_k(u_{k},A)\ge \alpha\int_A W(x, e(u))\dx\quad\forall A\in\A(U)\,,
	\end{equation*}
with $W$ given by \eqref{def:W}.
	\end{proposition}
\begin{proof} Let $(u_k)$ and $u$ be as in the statement. By Proposition \ref{prop:compactness} $u\in GSBD^p(U)$. Let $A\in\A(U)$ be fixed. Up to extracting a subsequence we may assume, by Theorem \ref{thm:conv-in-sobolev},  that 
	$$W'(x,M)=W''(x,M)=W(x,M)\,,$$
	with $W'$, $W''$ defined as in \eqref{eq:cell-formula-inf} and \eqref{eq:cell-formula-sup}.
For every $k\in\N$ let $\mu_k$ be the Radon measure on $(U, \mathcal B(U))$ given by
\begin{equation}\label{def-muk}
\mu_k(A)\defas F_k(u,A)\,,\quad \forall A \in \mathcal{B}(U)\,.
\end{equation}
As $\mu_k(A)\le C$ by \cite[Theorem 1.59]{AFP} we deduce the existence of a subsequence, not relabelled, and of  a Radon measure $\mu$ on $(A, \mathcal B(A))$ such that 
\begin{equation}\label{compactness+lsc}
\mu_k\stackrel{	*}{\rightharpoonup}\mu\quad \text{ and }\quad \liminf_{k\to+\infty}\mu_k(A)\ge \mu(A)\,.
\end{equation}
By Radon-Nikodym's Theorem there exist two measures $\mu^a,\mu^s$ with $\mu^a\ll \L^n$ and  $\mu^s\perp \L^n$, and a function  $h\in L^1(A)$  such that  $\mu=\mu^a+\mu^s$ and $\mu^a=h \L^n$.
This together with \eqref{compactness+lsc} imply that 
\begin{equation*}
\liminf_{k\to+\infty}F_k(u_k,A)\ge\int_A h(x)\dx\,.
\end{equation*}
Hence to conclude we need to show that 
\begin{equation}\label{claim}
h(x)\ge \alpha W(x,e(u(x)))\quad \text{ for a.e. }x\in U\,.
\end{equation}
with  $W$ as in \eqref{def:W}
For $i\in\N$ fixed  let $f_i(t)=\alpha_it\wedge\beta_i$ be as in \eqref{eq:est-f}.  Then it is enough to show that 
\begin{equation}\label{claim_i}
	h(x)\ge \alpha_i W(x,e(u(x)))\quad \text{ for a.e. }x\in U\,,
\end{equation}
We divide the proof of \eqref{claim_i} into four steps.\\

\noindent
\step 1
In this step we show that for a.e. $x_0\in U$ there exists a sequence $(k_j,r_j)\to(+\infty,0)$  as $j\to+\infty$ such that setting $\delta_j\defas\frac{\e_{k_j}}{r_j}$,
\begin{equation}\label{def:blow-up}
	u_{k}^{r}(v)\defas\frac{u_k(x_0+r v)-u_k(x_0)}{r}\quad \text{and}\quad
	W_k^r(x,M)\defas W(x_0+r x,M)\,,
\end{equation}
there hold
\begin{equation}\label{limit}
	h(x_0)\ge
	\lim_{j\to+\infty}\frac{1}{r_j} \frac{1}{\delta_{j}}\int_{ Q}f_i\Big(r_j{\delta_{j}}W_{k_j}^{r_j}(\cdot,e(u_{k_j}^{r_j}))*\rho_{\delta_{j}}(x)\Big)\dx\,,
\end{equation}
and 
\begin{equation}\label{approx-grad}
	u_{k_j}^{r_j}\to \nabla u(x_0)(\cdot)\quad \text{in measure on }Q\,.
\end{equation}
By
Besicovitch differentiation theorem  and \cite[Corollary 5.2]{CCS} we have that  for a.e. $x_0\in U$ the following hold:
\begin{equation}\label{besicovitch}
h(x_0)=\lim_{r\searrow0^+}\frac{\mu(\overline Q_r(x_0))}{|Q_r(x_0)|}\,,
\end{equation}
\begin{equation}\label{approx-gradient}
	\lim_{r\searrow0^+}\frac{1}{r^n}\L^n\left(\left\{y\in Q_r(x_0)\colon\frac{|u(y)-u(x_0)-\nabla u(x_0)(y-x_0)|}{|y-x_0|}>\delta\right\}\right)=0\quad\forall\delta>0\,.
\end{equation}
We fix $x_0\in \O$ for which \eqref{besicovitch} and \eqref{approx-gradient} hold. By \cite[Proposition 1.62]{AFP}  we have 
\begin{equation*}
	\mu(\overline Q_r(x_0))\ge\limsup_{k\to+\infty}\mu_k(\overline Q_r(x_0))
	\,,
\end{equation*}
for every $r>0$, which together with \eqref{besicovitch} yield
\begin{equation}\label{doppiolimite}
		h(x_0)\ge\lim_{r\searrow0^+}\limsup_{k\to+\infty}
	\frac{\mu_k(\overline Q_r(x_0))}{|Q_r(x_0)|}\,.
\end{equation}
Moreover from \eqref{def-muk} and the change of variable $x=x_0+r x'$ we get
 \begin{equation}\label{change-var1}
\begin{split}
\mu_k(\overline Q_r(x_0))&=\frac{1}{\e_k}\int_{\overline Q_r(x_0)}f\Big(\e_k W_k(\cdot,e(u_k))*\rho_k(x)\Big)\dx\\
&=\frac{r^n}{\e_k}\int_{\overline Q}f\Big(\e_k W_k(\cdot,e(u_k))*\rho_k(x_0+rx)\Big)\dx\,.
\end{split}
 \end{equation}
From \eqref{def:blow-up} and the  change of variable $y=ry'$ we may deduce that 
\begin{equation}\label{change-var2}
W_k(\cdot,e(u_k))*\rho_k(x_0+rx)=W_{k}^r(\cdot,e(u_{k}^r))*\rho_{\frac{\e_k}r}(x)\,,
\end{equation}
 as $k\to+\infty$.
Gathering together \eqref{doppiolimite}, \eqref{change-var1}, \eqref{change-var2} and using \eqref{eq:est-f} we obtain
\begin{equation*}
	h(x_0)\ge
\lim_{r\searrow0^+}\limsup_{k\to+\infty}\frac{1}{r} \frac{r}{\e_k}\int_{\overline Q}f_i\Big(r{\frac{\e_k}r}W_{k}^r(\cdot,e(u_{k}^r))*\rho_{\frac{\e_k}r}(x)\Big)\dx\,.
\end{equation*}
Furthermore from \eqref{approx-gradient} and the fact that $u_k$ converges to $u$ in measure we can deduce that 
\begin{equation*}
\lim_{r\to0}\lim_{k\to+\infty}\L^n\left(\left\{v\in Q\colon|u_{k}^r(v)-\nabla u(x_0)(v)|>\delta\right\}\right)=0\quad\forall\delta>0\,.
\end{equation*}
We now can find a subsequence $(k_j,r_j)\to(+\infty,0)$  as $j\to+\infty$ 
for which \eqref{limit} and \eqref{approx-grad} hold with $\delta_j\defas \tfrac{\e_{k_j}}{r_j}$
and the proof of step 1 is conluded.\\
\step 2  In this step we show that for any $0<\zeta<1$ and  a.e. $x_0\in U$   there exist  $(\bar u_j)\subset GSBV^p(Q;\R^n)$ and  $c_0>0$ independent of $j$ such that 
\begin{equation}\label{claim0}
	\lim_{j\to+\infty}\L^n\{\bar u_j\ne u_{k_j}^{r_j}\}=0;
\end{equation}
\begin{equation}\label{claimI}
	\bar u_j\to\nabla u(x_0)(\cdot)\quad \text{ in measure on } Q\,;
\end{equation}
\begin{equation}\label{claimII}
	\mathcal{H}^{n-1}(J_{\bar u_j}\cap Q)\to 0\,;
\end{equation}
\begin{equation}\label{claimIII}
	\int_{Q_{1- \zeta
		}(0)}|e(\bar u_j)|^p\dx\le c_0
	\,.
\end{equation}
By step 1 we have that  $u_{k_j}^{r_j}$ converges in measure  to $\nabla u(x_0)(\cdot)$ in $Q$ as $j\to +\infty$ and for $j$ large enough it satisfies 
\begin{equation}\label{boundedness}
	\frac{1}{r_j} \frac{1}{\delta_{j}}\int_{ Q}f_i\Big(r_j{\delta_{j}}W_{k_j}^{r_j}(\cdot,e(u_{k_j}^{r_j}))*\rho_{\delta_{j}}(x)\Big)\dx\le C\,.
\end{equation}
Next we fix $\eta\in(0,1)$ such that $Q_\eta(0)\subset\subset S$ and let $m_\eta$ and $f_i^\eta$ be as in \eqref{def:m-eta,f-eta}. Then we get
\begin{equation}\label{claim1}
\int_{ Q}f_i\Big({r_j\delta_{j}}W_{k_j}^{r_j}(\cdot,e(u_{k_j}^{r_j}))*\rho_{\delta_{j}}(x)\Big)\dx
 \ge
\int_{ Q}f_i^\eta\Big({r_j\delta_{j}}\strokedint_{Q_{\eta\delta_{j}}(x)}
 W_{k_j}^{r_j}(y,e(u_{k_j}^{r_j}))\dy\Big)\dx\,.
\end{equation}
 We  define the sets
\begin{equation*}
	A_j^1\defas\Biggl\{x\in Q\colon r_j\delta_{j}\strokedint_{Q_{\eta\delta_{j}}(x)}
	W_{k_j}^{r_j}(y,e(u_{k_j}^{r_j}))\dy\ge \frac{\beta_i}{\alpha_im_\eta\eta^{2n}}
	\Biggr\}\,,
\end{equation*}
\begin{equation*}
	A_j^2\defas\Biggl\{x\in Q\colon \dist(x,A_j^1)\le (1-\eta)\delta_{j}
	\Biggr\}\,.
\end{equation*}
Then arguing as in the proof of Proposition \ref{prop:compactness} we find that
\begin{equation}\label{containment}
A_j^1\subset	A_j^2\subset\Biggl\{x\in Q\colon r_j\delta_{j}\strokedint_{Q_{\delta_{j}}(x)}
	W_{k_j}^{r_j}(y,e(u_{k_j}^{r_j}))\dy\ge \frac{\beta_i}{\alpha_im_\eta\eta^{n}}
	\Biggr\}\,.
\end{equation}
 \eqref{containment} together with \eqref{boundedness} and \eqref{claim1} imply that (for $j$ large enough)
\begin{equation}\label{claim2}
C
\ge\frac{\beta_i}{r_j\delta_{j}}\L^n(A_j^2)=\frac{\beta_i}{\e_j}\L^n(A_j^2)
\,.
\end{equation}
By the coarea formula and the mean value theorem we can find $t_j\in(0,(1-\eta)\delta_{j})$ such that setting $A_j^3\defas\{\dist(\cdot,A_j^1)\le t_j
\}\subset A_j^2$
\begin{equation}\label{claim3}
\L^n(A_j^2)\ge (1-\eta)\delta_{j}\H^{n-1}(\partial A_j^3 )\,.
\end{equation}
We finally define 
\begin{equation*}
\bar u_j(x)\defas\begin{cases}
0&\text{ if } x\in A_j^3,\\
u_{k_j}^{r_j}&\text{ otherwise in }Q.
\end{cases}
\end{equation*}
Recall that, by definition, $\frac{\e_j}{\delta_j}\to 0$. With this, as a consequence of \eqref{claim2} and \eqref{claim3}  we have that both $\L^n( A_j^3)$ and $\H^{n-1}(\partial A_j^3)=\H^{n-1}(J_{\bar u_j})$ converge to $0$ as $j\to+\infty$. Hence $\bar u_j\subset GSBV^p(B_1(0);\R^n)$ and  $\bar u_j-u_{k_j}^{r_j}\to0$ in measure on $Q$ which combined with \eqref{approx-grad} yield  $\bar u_j\to\nabla u(x_0)(\cdot)$  in measure on $Q$. 
It remains to show \eqref{claimIII}.
To this aim notice that arguing exactly as in the proof of Proposition \ref{prop:compactness} one can find $K(n)\ge 1$ such that  for every $x\in \overline Q$
\begin{equation}\label{claim4}
	c_1{r_j\delta_{j}}\strokedint_{Q_{\eta\delta_{j}}(x)}
	|e(\bar u_{j}(y))|^p	\dy\le K \frac{\beta_i}{\alpha_im_\eta\eta^n}\,.
\end{equation}
Next from \eqref{claim1}, the monotonicity of $f_i^\eta$ we infer
\begin{equation}\label{eq:lets-believe}
	\begin{split}
	 \frac{1}{r_j\delta_{j}}		\int_{ Q}f_i\Big({r_j\delta_{j}}W_{k_j}^{r_j}(\cdot,e(u_{k_j}^{r_j}))*\rho_{\delta_{j}}(x)\Big)\dx&
		\ge \frac{1}{r_j\delta_{j}}
		\int_{ Q}f_i^\eta\Big(c_1{r_j\delta_{j}}\strokedint_{Q_{\eta\delta_{j}}(x)}
		|e(\bar u_{j}(y))|^p	\dy\Big)\dx\\
		&\ge c_1 \frac{\alpha_im_\eta\eta^n}{K} 	\int_{ Q}
		\strokedint_{Q_{\eta\delta_{j}}(x)}
		|e(\bar u_{j}(y))|^p	\dy\dx
		\,
	\end{split}
\end{equation}
where the last inequality follows from \eqref{claim4} and fact that $f_i^\eta(t)\ge  \frac{\alpha_im_\eta\eta^n}{K} t$  when $t\le K \frac{\beta_i}{m_\eta\eta^n}$.
Finally, for a fixed $0<\zeta<1$, arguing exactly as for \eqref{eq:comp7} we get
\begin{equation}\label{eq:c-var}
	\begin{split}
			\int_{ Q}
		\strokedint_{Q_{\eta\delta_{j}}(x)}
		|e(\bar u_{j}(y))|^p	\dy\dx\ge \int_{Q_{1-\zeta}(0)} |e(\bar u_j(x))|^p\dx\,
	\end{split}
\end{equation}
when $j$ is sufficiently large. Eventually gathering together \eqref{boundedness}, \eqref{eq:lets-believe}, and \eqref{eq:c-var}, we deduce \eqref{claimIII} with $c_0\defas \frac{CK}{c_1\alpha_im_\eta\eta^n}$.
\\
\step 3  In this step show that for a.e. $x_0\in U$ there exists a sequence $(w_j)\subset W^{1,p}(Q;\R^n)$ such that:
\begin{equation}\label{1}
	(|\nabla w_j|^p)\quad \text{ is equi-integrable; }
\end{equation}
\begin{equation}\label{2}
	\lim_{j\to+\infty}\L^n(\left\{w_j\ne\bar u_j \right\}=0\,;
\end{equation}
\begin{equation}\label{3}
	\lim_{j\to+\infty}	\|w_j-\nabla u(x_0)(\cdot) \|_{ L^p(Q)}=0\,;
\end{equation}
\begin{equation}\label{next-step}
	h(x_0)\ge  \liminf_{j\to+\infty}\frac{1}{r_j\delta_{j}}\int_{ Q_{1-\zeta}(0)}f_i\Big({r_j\delta_{j}}W_{k_j}^{r_j}(\cdot,e(w_j))*\rho_{\delta_{j}}(x)\Big)\dx\quad \forall i\in\N\,.
\end{equation}
From step 2 we can apply  \cite[Lemma 5.1]{FPS} to the sequence $\bar u_j$ and get the existence of $(w_j)\subset W^{1,p}(Q;\R^n)$ that satisfies \eqref{1}--\eqref{3}. Moreover recalling \ref{hyp:growth-W} and the equi-integrability of  $(|\nabla w_j|^p)$ we have that $W_{k_j}^{r_j}(x,e(w_j))$ is equi-integrable as well, while from the inclusion
\begin{equation*}
E_j\defas	\left\{ e(w_j)\ne e( u_{k_j}^{r_j} )\right\}\subset \{w_j\ne u_{k_j}^{r_j}\}\subset
	\left\{ w_j\ne \bar u_j \right\}\cup\{\bar u_j\ne u^{r_j}_{k_j}\}
	\,,
\end{equation*}
it follows that $\L^n\left(	\left\{ e(w_j)\ne e( u_{k_j}^{r_j} )\right\}\right)\to0$.
Thus, we can apply Lemma \ref{technical-lemma1} with  $g_j=W_{k_j}^{r_j}(x,e(w_j))$, and $E_j=\left\{ e(w_j)\ne e( u_{k_j}^{r_j} )\right\}$, and deduce that 
\begin{equation}\label{reminder}
\int_{E_j}W_{k_j}^{r_j}(\cdot, e(w_j))*\rho_{\delta_{j}}(x)\dx\to0\,.
\end{equation}
Using that $f_i(t)\le \alpha_it$ we obtain the following estimate
\begin{equation*}
\begin{split}\frac{1}{r_j\delta_{j}}&
	\int_{ Q_{1-\zeta}(0)}f_i\Big({r_j}\delta_jW_{k_j}^{r_j}(\cdot,e(w_j))*\rho_{\delta_{j}}(x)\Big)\dx\\
&\le\frac{1}{r_j\delta_{j}}
\int_{ Q_{1-\zeta}(0)}f_i\Big({r_j\delta_{j}}W_{k_j}^{r_j}(\cdot,e( u_{k_j}^{r_j}))*\rho_{\delta_{j}}(x)\Big)\dx
+\alpha_i\int_{E_j}W_{k_j}^{r_j}(\cdot,e(w_j))*\rho_{\delta_{j}}(x)\dx\,.
\end{split}
\end{equation*}
Passing to the limit as $j\to+\infty$ in the above inequality and using \eqref{limit} and \eqref{reminder} we infer  \eqref{next-step}.\\
\step 4  In this step we show that for a.e. $x_0\in U$
\begin{equation}\label{last-step}
	\liminf_{j\to+\infty}\frac{1}{r_j\delta_{j}}\int_{ Q_{1-\zeta}(0)}f_i\Big({r_j\delta_{j}}W_{k_j}^{r_j}(\cdot,e(w_j))*\rho_{\delta_{j}}(x)\Big)\dx\ge \alpha_iW(x_0,e(u(x_0)))\quad\forall i\in\N\,.
\end{equation}
We define the following partition
\begin{equation*}
	B_j^1\defas \left\{x\in Q_{1-\zeta}(0)\colon {r_j\delta_{j}}W_{k_j}^{r_j}(\cdot,e(w_j))*\rho_{\delta_{j}}(x)\ge \frac{\beta_i}{\alpha_i}
	\right\}\,,\quad B_j^2\defas  Q_{1-\eta}(0)\setminus 	B_j^1\,.
\end{equation*}
Then for $j$ large enough there holds 
\begin{equation*}
\L^n(	B_j^1)\le Cr_j\delta_{j}\to0\,,
\end{equation*}
which, together with Lemma \ref{technical-lemma1}, implies
\begin{equation}\label{reminder-bis}
\int_{B_j^1 }W_{k_j}^{r_j}(\cdot,e(w_j))*\rho_{\delta_{j}}(x)\dx\to0\,.
\end{equation}
Since $f_i(t)=\alpha_it$ when $t\le \frac{\beta_i}{\alpha_i}$, we then have 
\begin{equation}\label{almost-the-end}
\frac{1}{r_j\delta_{j}}\int_{ Q_{1-\zeta}(0)}f_i\Big({r_j\delta_{j}}W_{k_j}^{r_j}(\cdot,e(w_j))*\rho_{\delta_{j}}(x)\Big)\dx\ge 
{\alpha_i}\int_{B_j^2}W_{k_j}^{r_j}(\cdot,e(w_j))*\rho_{\delta_{j}}(x)\dx\,.
\end{equation}
Now, taking the liminf as $j\to+\infty$ in \eqref{almost-the-end}, and using \eqref{reminder-bis}, we get
\begin{equation}\label{eq:step4-0}
	\begin{split}
\liminf_{j\to+\infty}	\frac{1}{r_j\delta_{j}}\int_{ Q_{1-\zeta}(0)}f_i\Big({r_j\delta_{j}}W_{k_j}^{r_j}(\cdot,e(w_j))*\rho_{\delta_{j}}(x)\Big)\dx\ge \liminf_{j\to+\infty}
{\alpha_i}\int_{Q_{1-\zeta}(0)}W_{k_j}^{r_j}(\cdot,e(w_j))*\rho_{\delta_{j}}(x)\dx\,.
	\end{split}
\end{equation}
From this, applying Lemma \ref{technical-lemma2} with $g_j=W_{k_j}^{r_j}(x,e(w_j))$ we have
\begin{equation}\label{eq:step4-1}
\liminf_{j\to+\infty}
{\alpha_i}\int_{Q_{1-\zeta}(0)}W_{k_j}^{r_j}(\cdot,e(w_j))*\rho_{\delta_{j}}(x)\dx\ge{\alpha_i} \liminf_{j\to+\infty}\int_{Q_{1-\zeta}(0)}W_{k_j}(x_0+r_jx,e(w_j))\dx\,.
\end{equation}
Next we modify $w_j$ so that it coincides with $\nabla u(x_0)(\cdot)$ on $\partial Q_{1-\zeta}(0)$ without essentially increasing the energy. This can be achieved by relying on the following Fundamental Estimate than can be proved with standard arguments:  there exist $C(\gamma)$ and a sequence $(\overline w_j)\subset W^{1,p}(Q_{1-\zeta}(0);\R^d)$ with $\overline w_i=\nabla u(x_0)(\cdot)$ in a neighbourhood of $\partial Q_{1-\zeta}(0)$ such that
\begin{equation}\label{eq:step4-1-1}
	\begin{split}
	\int_{Q_{1-\zeta}(0)}W_{k_j}(x_0+r_jx,e(\overline w_j))\dx
&	\le (1+\gamma)\int_{Q_{1-\zeta}(0)}W_{k_j}(x_0+r_jx,e(w_j))\dx\\&+
(1+\gamma)	\int_{Q_{1-\zeta}(0)\setminus Q_{1-\zeta-\gamma}(0)}W_{k_j}(x_0+r_jx,e(u(x_0)))\dx	\\&
+ C(\gamma)\|w_j-\nabla u(x_0)(\cdot)\|^p_{L^p(Q_{1-\zeta}(0))}+\gamma
\,.
	\end{split}
\end{equation}
By \eqref{3} we know that $w_j$ converges to $\nabla u(x_0)(\cdot)$ in $L^p(Q)$, moreover  from \eqref{hyp:growth-W} there holds
\begin{equation*}\label{eq:step4-1-2}
	\begin{split}
	\int_{Q_{1-\zeta-\gamma}(0)}W_{k_j}(x_0+r_jx,e(u(x_0)))\dx 
	&\le c_2(|e(u(x_0))|^p+1) \L^n(
{Q_{1-\zeta}(0)\setminus Q_{1-\zeta-\gamma}(0)  }) \\
& \le c_2(|e(u(x_0))|^p+1)n\gamma\,.
	\end{split}
\end{equation*}
This fact and \eqref{eq:step4-1-1} imply that
\begin{equation}\label{eq:step4-2}
	\begin{split}
		\liminf_{j\to+\infty}&\int_{Q_{1-\zeta}(0)}W_{k_j}(x_0+r_jx,e(w_j))\dx\\
		&\ge  \frac{1}{1+\gamma} 	\liminf_{j\to+\infty}
		\int_{Q_{1-\zeta}(0)}W_{k_j}(x_0+r_jx,e(\overline w_j))\dx
		-c_2(|e(u(x_0))|^p+1)n\gamma-\frac{\gamma}{1+\gamma}\,.
	\end{split}
\end{equation}
We now set $\tilde w_j(x)\defas r_j\overline w_j((x-x_0)/r_j)$, which is admissible for $\m_k(u_{e(u(x_0))},Q_{(1-\zeta)r_j}(0))$ in \eqref{min-prob}.
 Hence, by a  change of variable in \eqref{eq:step4-2} we obtain
\begin{equation}\label{eq:step4-3}
	\begin{split}
	\liminf_{j\to+\infty}{\alpha_i}
\int_{Q_{1-\zeta}(0)}W_{k_j}(x_0+r_jx,e(\overline w_j))\dx&\ge 
\liminf_{j\to+\infty}\frac{\alpha_i}{r_j^n}\int_{Q_{(1-\zeta)r_j}(x_0)}W_{k_j}(x,e(\tilde w_j))\dx\\
&\ge\liminf_{j\to+\infty}\alpha_i\frac{\m_{k_j}(u_{e(u(x_0))},Q_{(1-\zeta)r_j}(x_0))}{r_j^n}\\
&
=(1-\zeta)\alpha_iW(x_0,e(u(x_0)))\,.
	\end{split}
\end{equation}
Gathering together \eqref{eq:step4-2} and \eqref{eq:step4-3} we deduce
\begin{equation*}
	\liminf_{j\to+\infty}{\alpha_i}\int_{Q_{1-\zeta}(0)}W_{k_j}(x_0+r_jx,e(w_j))\dx
	\ge \frac{(1-\zeta)^n}{1+\gamma}\alpha_iW(x_0,e(u(x_0)))	-C\left(\gamma+ \frac{\gamma}{1+\gamma}\right)\,.
\end{equation*}
With this, \eqref{eq:step4-0}, and \eqref{eq:step4-1}, we eventually deduce \eqref{last-step} by arbitrariness of $\zeta$ and $\gamma$.\\
\textit{Conclusion:} from step 3 and step 4 we deduce the validity of \eqref{claim_i} and the proof is concluded.
\end{proof}

\begin{remark}\label{cortesani}
We observe en passant that Proposition \ref{blow-up} indeed holds also for a sequence of functionals
\[
	F_k(u,A)\defas \frac1{\e_k}\int_A f_k\Big(\e_kW_k(\cdot,e(u))*\rho_k(x)
	\Big)\dx \,,\quad \text{for } u\in W^{1,p}(U)\,,\ A\subset U\,,
\]
provided the functions $f_k$ satisfy an estimate of the form
\[
f_k(t)\ge \alpha_k t \wedge \beta
\]
for all $t\in [0,+\infty)$, where $\beta$ is a uniform constant and $\alpha=\lim_{k\to +\infty}\alpha_k$.
\end{remark}
\begin{proposition}[Lower bound: surface contribution]\label{prop:lb-surf} 
  Let $(u_k)\subset L^0(U;\R^n)$ be a sequence  that converges to  in measure to $u\in  L^0(U;\R^n)$. Assume moreover  that $F_k(u_k)\le C$.
  Then there holds 
\begin{equation*}
	\liminf_{k\to+\infty}F_k(u_k,A)\ge \beta 	\int_{J_u\cap A}\phi_\rho(\nu_u)\dHn\quad \forall A\in \A(U)\,.
\end{equation*}
\end{proposition}
\begin{proof} 
	Let $(u_k)$ and $u$ be as in the statement, so that by Proposition \ref{prop:compactness} $u\in GSBD^p(U)$. Let $A\in\A(U)$ be fixed.
 We claim that it suffices to show that for any $\xi\in\S^{n-1}$ fixed there holds
	\begin{equation}\label{eq:liminf-xi}
		\liminf_{k\to+\infty}F_k^\xi(u_k,A)
		\ge \beta \int_{J_u^\xi\cap A}\mu_\xi|\langle\nu_u,\xi\rangle|\dHn\,,
	\end{equation}
	with 
	\begin{equation*}
		F_k^\xi(u_k,A)\defas\frac{1}{\e_k}\int_A f\Big(c_1  \e_k   |\langle e(u_k)\xi,\xi\rangle|^p  *\rho_k(x)\Big)\dx\,,
	\end{equation*}
	 $$J_u^\xi\defas\{x\in J_u\colon \langle u^+(x)-u^-(x),\xi\rangle\ne0\}\quad\text{and}\quad \mu_\xi\defas\H^1(\{x\in S\colon x=t\xi\ \text{ for }\ t\in\R\})\,.$$ 
Indeed, assume for the moment \eqref{eq:liminf-xi} holds true. Then
 \eqref{hyp:growth-W} gives
		\begin{equation*}
		W_k(x,e(u_k))\ge c_1|e(u_k)|^p\ge c_1 |\langle(e(u_k))\xi,\xi\rangle|^p \,.
	\end{equation*}
Since $f$ is nondecreasing, the above implies
	\begin{equation*}
\liminf_{k\to+\infty}	F_k(u_k,A)\ge\liminf_{k\to+\infty} F_k^\xi(u_k,A)	\ge \beta \int_{J_u^\xi\cap A}\mu_\xi|\langle\nu_u,\xi\rangle|\dHn=\beta\int_{J_u\cap A}\varphi_\xi
\dHn\,,
	\end{equation*}
with $\varphi_\xi\colon J_u\to[0,+\infty]$ given by 
\begin{equation*}
	\varphi_\xi(x)\defas \begin{cases}
		\mu_{\xi}|\langle\nu_u(x),\xi\rangle|&\text{if }x\in J_u^{\xi}\,,\\
		0&\text{otherwise}\,.
	\end{cases}
\end{equation*}
Now let $(\xi_h)\subset\S^{n-1}$ be a dense subset, in this way by \cite[Proposition 1.16]{Braides} it holds 
\begin{equation*}
\liminf_{k\to+\infty}	F_k(u_k,A)\ge \beta\int_{J_u\cap A}\sup_h\varphi_{\xi_h}
\dHn\,.
\end{equation*}
On the other hand by \cite[Lemma 4.5]{FarSciSol} we have 
	\begin{equation}\label{eq:sup-xi}
		\phi_\rho(\nu)=\sup_{\xi\in\S^{n-1}}\mu_\xi |\langle\nu,\xi\rangle| \,,
	\end{equation}
which in turn implies $\phi_\rho(\nu_u(x))=\sup_h\varphi_{\xi_h}(x)$ and the thesis follows. It remains to show \eqref{eq:liminf-xi} for which 	we will argue by slicing.  
 As the set $S$ is convex for $\delta\in(0,1)$ fixed we can find $r=r(\delta,S)$ such that the cylinder 
	\begin{equation*}
		C_\xi^{r,\delta}\defas R_\xi(Q'_r(0)\times \big(-{\mu_\xi\delta}/2,{\mu_\xi\delta}/2\big))
		\subset\subset S\,,
	\end{equation*}
	where $R_\xi\in SO(n)$ is such that $R_\xi e_n=\xi$ (see \ref{Rn}).  Let now $m_\xi\defas\min_{x\in \overline{C}_\xi^{r,\delta}}\rho(x)$ and 
	\begin{equation*}
		C_{\xi,k}^{r,\delta}(x)\defas  	\e_kC_\xi^{r,\delta}+x\,.
	\end{equation*}
Next for any $x\in A$ we denote by $\hat x_\xi$ the projection of $x$ onto  $\Pi^\xi\defas\{y\in\R^n\colon \langle y,\xi\rangle=0\}$ and
\begin{equation*}
	I_\xi\defas\{\tau\in\R\colon \hat x_\xi+\tau\xi\in A \}\,.
\end{equation*}
	Thus we have 
	\begin{equation}\label{eq:slicing1}
		\begin{split}
			F_k^\xi(u_k,A)
		&	\ge
		\frac{1}{\e_k}	\int_{A}f\biggl(\frac{c_1m_\xi}{\e_k^{n-1}}
			\int_{C_{\xi,k}^{r,\delta}(x)}
			|\langle e(u_k(z))\xi,\xi\rangle|^p  \dz
			\biggr)\dx\\
			&=\frac{1}{\e_k}\int_{\Pi^\xi}\int_{I_\xi}
			f\biggl(\frac{c_1m_\xi}{\e_k^{n-1}}
			\int_{C_{\xi,k}^{r,\delta}(\hat x_\xi+\tau\xi)}
			|\langle e(u_k(z))\xi,\xi\rangle|^p  \dz
			\biggr)	{\rm d}\tau\dHn(\hat x_\xi)\
			\,,
		\end{split}
	\end{equation}
where the last equality follows by Fubini's Theorem. Noticing that $f$ is concave and using the change of variable $z=\hat x_\xi+s\xi+\e_krz'$ with 
	$$z'\in Q'_{\xi}:=R_\xi\big(Q'\times\{0\}\big)\quad\text{and}\quad s\in \bigg(\tau-\frac{\mu_\xi\delta\e_k}2,\tau+\frac{\mu_\xi\delta\e_k}2\bigg)\,,$$
	together with Jensen's inequality yield
	\begin{equation}\label{eq:sliing2}
		\begin{split}
			f\biggl(\frac{c_1m_\xi}{\e_k^{n-1}}\int_{C_{\xi,k}^{r,\delta}(\hat x_\xi+\tau\xi)}
			|\langle e(u_k(z))\xi,\xi\rangle|^p  \dz\biggr)
			&\ge\strokedint_{Q'_{\xi}}
			\tilde f\biggl( \int_{\tau-\frac{\mu_\xi\delta\e_k}{2}}^{\tau+\frac{\mu_\xi\delta\e_k}{2}}
			|\langle e(u_k(\hat x_\xi+\e_krz'+s\xi))\xi,\xi\rangle|^p \ds\biggr)\dz'\\
			&=
			\strokedint_{Q'_{\xi}}
			\tilde f\biggl( \int_{\tau-\frac{\mu_\xi\delta\e_k}{2}}^{\tau+\frac{\mu_\xi\delta\e_k}{2}}
			\Big|\frac{\partial}{\partial s}w_{\xi,k}(\hat x_\xi, z',s)\Big|^p \ds\biggr)\dz'
		\end{split}
	\end{equation}
	with $\tilde f(t):=f\big(\frac{c_1m_\xi}{\e_k^{n-1}}t\big)$ and $w_{\xi,k}(\hat x_\xi, z',s)\defas \langle u_k(\hat x_\xi+\e_krz'+s\xi),\xi\rangle$. 
	Observe now that applying Corollary \ref{cor:convL0} and Fubini's Theorem to the functions $w_{\xi,k}(\hat x_\xi, z',s)$ we have that, for a.e. $(\hat x_\xi, z')\in \Pi_\xi\times Q'_{\xi}$ the functions $s\mapsto w_{\xi,k}(\hat x_\xi, z',s)$ converge to the section
	$u^{\hat x_\xi}(s):=u(\hat x_\xi+s\xi)\cdot \xi$ in measure on $I_\xi$.
	Further, gathering together \eqref{eq:slicing1} and \eqref{eq:sliing2}, and exchanging the order of integration it holds 
	\begin{equation}\label{eq:slicing3}
		F_k^\xi(u_k,A)
		\ge
		\int_{\Pi^\xi}\strokedint_{Q'_\xi}\bigg(\frac{1}{\e_k}\int_{ I_\xi }
		\tilde f\biggl(\int_{\tau-\frac{\mu_\xi\delta\e_k}{2}}^{\tau+\frac{\mu_\xi\delta\e_k}{2}}
		|\dot w^{\hat x_\xi, z'}_{\xi,k}(s)|^p \ds\biggr)
		{\rm d}\tau
		\bigg)  \dz'              \dHn(\hat x_\xi)\,,
	\end{equation}
	where the shortcut $w^{\hat x_\xi, z'}_{\xi,k}(s)$ denotes the function $s\mapsto w_{\xi,k}(\hat x_\xi, z',s)$ for fixed $(\hat x_\xi, z')$. By Theorem \ref{thm:1dim} we get 
	\begin{equation}\label{eq:1dimlimit}
		\liminf_{k\to+\infty}
		\frac{1}{\e_k}\int_{ I_\xi }
		\tilde f\biggl(\int_{\tau-\frac{\mu_\xi\delta\e_k}{2}}^{\tau+\frac{\mu_\xi\delta\e_k}{2}}
		|\dot w^{\hat x_\xi, z'}_{\xi,k}(s)|^p \ds\biggr)
		{\rm d}\tau 
		\ge \beta\delta \mu_\xi\#(J_{u^{\hat x_\xi}}\cap I_\xi )\,.
	\end{equation}
	Combining \eqref{eq:slicing3} with \eqref{eq:1dimlimit} we finally obtain
	\begin{equation*}
		\liminf_{k\to+\infty}	F_k^\xi(u_k,A)
		\ge\delta\beta  \int_{\Pi^\xi} \mu_\xi   \#(J_{u^{\hat x_\xi} }\cap I_\xi)   \dHn(\hat x_\xi) =
		\delta \beta \int_{J_u^\xi\cap A}\mu_\xi|\langle\nu_u,\xi\rangle|\dHn\,.
	\end{equation*}
Eventually by the arbitrariness of $\delta$ we deduce \eqref{eq:liminf-xi}.
\end{proof}
With the help of Propositions \ref{blow-up} and \ref{prop:lb-surf} we can now prove the following lower bound.
\begin{proposition}[Lower-bound]\label{prop:lower-bound}
Let $F_k$ and $F$ be as in \eqref{F} and \eqref{def:limit-F} respectively. Let  $(u_k)\subset L^0(U;\R^n)$, $u\in GSBD^p(U;\R^n)$ be such that $u_k$ converges to $u$ in measure. Then there exists a subsequence, not relabelled,  such that 
\begin{equation*}
	\liminf_{k\to+\infty}F_k(u_k)\ge  F(u)\,.
\end{equation*}
\end{proposition}
\begin{proof}	Let $(u_k)$ and $u$ be as in the statement. 
		We assume without loss of generality that 
		\begin{equation*}
				\liminf_{k\to+\infty}F_k(u_k)=\lim_{k\to+\infty}F_k(u_k)<+\infty\,.
			\end{equation*}
		Thus in particular there exists $C>0$ such that 
		$
		F_k(u_k)\le C
		$ and by Proposition \ref{prop:compactness} it follows that $u\in GSBD^p(U)$.
 We define $\varphi_1,\varphi_2\colon U\to[0,+\infty)$ as follows
\begin{equation*}
\varphi_1(x)\defas\begin{cases}
\alpha W\big(x,e(u(x))\big)&\text{if }x\in U\setminus J_u\,,\\
0&\text{otherwise}\,,
\end{cases}
\end{equation*}
\begin{equation*}
	\varphi_2(x)\defas\begin{cases}
	\beta\phi_\rho(\nu_u(x))&\text{if }x\in  J_u\,,\\
		0&\text{otherwise}\,.
	\end{cases}
\end{equation*}
Moreover set $$\lambda\defas \L^n\res\Omega+\H^{n-1}\res J_u\,.$$
Now by invoking Propositions \ref{blow-up} and \ref{prop:lb-surf} we have that
\begin{equation*}
\liminf_{k\to+\infty}F_k(u_k,A)\ge \int_A\varphi_1\d\lambda\,,
\end{equation*}
and 
\begin{equation*}
	\liminf_{k\to+\infty}F_k(u_k,A)\ge \int_A\varphi_2\d\lambda\,,
\end{equation*}
for all $A\in  \A(U)$.
Then by \cite[Proposition 1.16]{Braides} we find that 
\begin{equation*}
	\liminf_{k\to+\infty}F_k(u_k,A)\ge \int_A\varphi_1\vee\varphi_2\d\lambda= F(u,A)\,,
\end{equation*}
from which in particular 
\begin{equation*}
	\liminf_{k\to+\infty}F_k(u_k)\ge \int_A\varphi_1\vee\varphi_2\d\lambda= F(u)\,.
\end{equation*}
\end{proof}
%
\section{Upper bound}\label{sec:ub}
In this section we prove the upper bound. 
\begin{proposition}\label{prop:upper-bound}
	Let $F_k$ and $F$ be as in \eqref{F} and \eqref{def:limit-F} respectively. Then for each $u\in L^0(U;\R^n)$ there is $(u_k)\subset L^0(U;\R^n)$ that converges in measure to $u$ and such that
	\begin{equation*}
		\limsup_{k\to\infty}F_k(u_k)\le F(u)\,.
	\end{equation*}
\end{proposition}
\begin{proof}
	Without loss of generality we assume  $F(u)<C$ so that $u\in GSBD^p(U)$. Moreover by Theorem \ref{thm:GSBD-density} we can assume that $u\in \mathcal W^\infty_{\rm pw}(U;\R^n)$ and that  $J_u$ is a connected $(n-1)$-rectifiable set compactly contained in $U$.  
	We fix  $U'\in\A$ with $U\subset\subset U'$ and  consider  an extension of $u$ on $U'$, not relabelled, such that  $u\in \mathcal W^\infty_{\rm pw}(U';\R^n)$. 
	Then by Theorem \ref{thm:conv-in-sobolev} and Remark \ref{rem:conv-in-sobolev} we can find $(v_k)\subset  W^{1,p}(U'\setminus J_u;\R^n)$ such that $v_k$ converges strongly to $u$ in $L^p(U'\setminus J_u;\R^n)$
and 
\begin{equation}\label{eq:recovery}
	\lim_{k\to\infty}E_k(v_k,U'\setminus J_u)= E(u,U'\setminus J_u)=\int_{U'}W(x, e(u))\dx\,.
\end{equation}
where the last equality clearly holds as $J_u$ is a null set. For every $h>0$ we set
\begin{equation*}
	(J_u)_h\defas\{x\in U\colon {\rm d}_S(x,J_u)<h\}\,,
\end{equation*}
so that for $h$ small enough $(J_u)_h\subset\subset U$. 
Fix now $0<\delta_k<<\e_k$ and take $\varphi_k\in C^\infty_c(U')$  cutoff between $(J_u)_{\delta_k}$ and $(J_u)_{2\delta_k}$. Next define $(u_k)\subset W^{1,p}(U';\R^n)$ as 
\begin{equation*}
u_k\defas v_k(1-\varphi_k)\to u\quad\text{strongly in }L^p(U'\setminus J_u;\R^n)\,,
\end{equation*} and in particular $u_k\to u$ in measure on $U'$.
 Then using that $u_k=v_k$ in $U'\setminus(J_u)_{2\delta_k}$ we have
 \begin{equation}\label{eq:limsup}
 	\begin{split}
F_k(u_k)\le F_k(v_k, U\setminus J_u)+ \beta\frac{\L^n((J_u)_{2\delta_k+\e_k})}{\e_k}\,.
 	\end{split}
 \end{equation}
Now invoking  \cite[Theorem 3.7]{Lussardi-Villa} we have
\begin{equation}\label{eq:limsup1}
	\lim_{k\to\infty}\frac{\L^n((J_u)_{2\delta_k+\e_k})}{\e_k}=\int_{J_u}\phi_\rho(\nu_u)\dHn\,.
\end{equation}
Moreover as  $f$ is increasing and satisfies \eqref{hyp:f} then $f(t)\le \hat\alpha t$ for all  $\hat\alpha>\alpha$. This together with the change of variable $y=x-\e_kz$, Fubini's theorem, and the change of variable $\hat x=x-\e_k z$ yield
 \begin{equation*}
 	\begin{split}
 	 F_k(v_k, U\setminus J_u)
 	 &\le \hat\alpha \int_{U\setminus J_u}\int_{\R^n}W_k(y,e(v_k))\rho_k(x-y)\dy\dx\\
 &	 = \hat{\alpha}\int_{U}\int_{\R^n}W_k\big(x-\e_kz, e(v_k(x-\e_k\cdot))\big)\rho(z)\dz\dx\\
 &= \hat{\alpha}\int_{\R^n}\rho(z)\int_{U} W_k\big(x-\e_kz, e(v_k(x-\e_k\cdot))\big)\dx\dz\\
 &\le\hat\alpha \int_{U'}W_k(x,e(v_k))\dx=E_k(v_k,U'\setminus J_u)\,.
 	\end{split}
 \end{equation*}
 Hence passing to the limit in $k$ in the above inequality and using  \eqref{eq:recovery} we get
 \begin{equation}\label{eq:limsup2}
\limsup_{k\to\infty}F_k(v_k, {U\setminus J_u})\le \hat\alpha E(u,U'\setminus J_u)=\hat\alpha\int_{U'}W(x, e(u))\dx\,,
 \end{equation}
 for all $\hat\alpha>\alpha$. Finally gathering together \eqref{eq:limsup}-\eqref{eq:limsup2} we obtain
 \begin{equation*}
\limsup_{k\to\infty}F_k(u_k)\le \hat{\alpha}\int_{U'}W(x, e(u))\dx+\beta \int_{J_u}\phi_\rho(\nu_u)\dHn\,.
 \end{equation*}
 Eventually by the arbitrariness of $U'$ and  $\hat\alpha$ we conclude. 
\end{proof}
\begin{remark}\label{rem:ub}
	If a lower order term $\int_U\psi(|u|)\dx$, is added to the energy, the density argument above can still be applied if $\psi$ complies with the assumptions of Theorem \ref{thm:GSBD-density}. Also observe that within the same assumptions, $\int_U\psi(|u_k|)\dx$ is equiintegrable whenever $(u_k)$ is converging in $L^p$. For $u_k$ and $u$ as in the proof above, this entails the convergence $\int_U\psi(|u_k|)\dx\to \int_U\psi(|u|)\dx$. 
\end{remark}
We are now in a position to prove Theorem \ref{thm:main-theorem}.
\begin{proof}[Proof of Theorem \ref{thm:main-theorem}]
Point $(i)$ follows by combining Propositions \ref{prop:lower-bound} and \ref{prop:upper-bound}, while $(ii)$ is a consequence of Proposition \ref{prop:compactness}.
\end{proof}
\section{Stochastic homogenisation}\label{sec:stoch-hom}
In this section we are concerned with the $\Gamma$-convergence analysis of the functionals $F_k$ when $W_k$ are random integrands of type
\begin{equation*}
W_k(\omega,y,M)=W\Big(\omega,\frac{y}{\delta_k},M\Big)\,,
\end{equation*}
with  $\omega$ belonging to the sample space $\O$ of a complete probability space $(\O,\T,P)$ and $\delta_k\searrow0$.
In order to do that we first give some definitions.
\begin{definition}[Group of $P$-preserving transformations]\label{def:group-P-pres-transf}
A group of $P$-preserving transformations on $(\O,\T,P)$ is a family $(\tau_z)_{z\in\Z^n}$ of mappings $\tau_z\colon\O\to\O$ satisfying the following:
\begin{enumerate}[label*=$(\alph*)$]
	\item (measurability) $\tau_z$ is $\T$-measurable for every $z\in\Z^n$;
	\item (invariance) $P(\tau_z(E))=P(E)$, for every $E\in \T$ and every $z\in\Z^n$;
	\item (group property) $\tau_0={\rm id}_\O$ and $\tau_{z+z'}=\tau_z\circ\tau_{z'}$ for every $z,z'\in\Z^n$.
\end{enumerate}
if in addition, every $(\tau_z)_{z\in\Z^n}$-invariant set (that is, every $E\in\T$ with $\tau_z(E)=E$ for every $z\in\Z^n$) has probability 0 or 1, then $(\tau_z)_{z\in\Z^n}$ is called ergodic.
\end{definition}
Let $a\defas(a_1,\dots,a_n),\,b\defas(b_1,\dots,b_n)\in\Z^n$ with $a_i<b_i$ for all $i\in\{1,\ldots,n\}$; we define the $n$-dimensional interval 
\begin{equation*}
	[a,b):=\{x\in\Z^n\colon a_i\le x_i<b_i\,\,\text{for}\,\,i=1,\ldots,n\}
\end{equation*}
and we set
\begin{equation*}
	\I_n:=\{[a,b)\colon a,b\in\Z^n\,,\, a_i<b_i\,\,\text{for}\,\,i=1,\ldots,n\}\,.
\end{equation*}
\begin{definition}[Subadditive process]\label{def:subadditive-proc}
	A discrete subadditive process with respect to a group $(\tau_z)_{z\in\Z^n}$ of $P$-preserving transformations on $(\O,\T,P)$ is a function $\mu\colon\O\times \I_n\to\R$ satisfying the following:
		\begin{enumerate}[label=$(\alph*)$]
		\item\label{subad:meas} (measurability) for every $A\in\I_n$ the function $\omega\mapsto\mu(\omega,A)$ is $\T$-measurable;
		\item\label{subad:cov} (covariance) for every $\omega\in\Omega$, $A\in\I_n$, and $z\in\Z^n$ we have $\mu(\omega,A+z)=\mu(\tau_z(\omega),A)$;
		\item\label{subad:sub} (subadditivity) for every $A\in\I_n$ and for every finite family $(A_i)_{i\in I}\subset\I_n$ of pairwise disjoint sets such that $A=\cup_{i\in I}A_i$, we have
		\begin{equation*}
			\mu(\omega,A)\le\sum_{i\in I}\mu(\omega,A_i)\quad\text{for every}\,\,\omega\in\Omega\,;
		\end{equation*}
		\item\label{subad:bound} (boundedness) there exists $c>0$ such that $0\le\mu(\omega,A)\le c\L^n(A)$ for every $\omega\in\Omega$ and $A\in\I_n$.
	\end{enumerate}
\end{definition}
\begin{definition}[Stationarity] \label{def:stationarity}
	Let $(\tau_z)_{z\in\Z^n}$ be a group of $P$-preserving transformations on $(\Omega,\T,P)$. 
	We  say that $W\colon\Omega\x\R^n\x\mathbb{M}^{n\times n}\to[0,+\infty)$ is \textit{stationary} with respect to $(\tau_z)_{z\in\Z^n}$ if
	\begin{equation*}
		W(\omega,x+z,M)=W(\tau_z(\omega),x,M)
	\end{equation*}
	for every $\omega\in\Omega$, $x\in\R^n$, $z\in\Z^n$ and $M\in\mathbb{M}^{n\times n}$. 	Moreover we say that a stationary random integrand $W$ is \textit{ergodic} if $(\tau_z)_{z\in\Z^n}$ is ergodic.
\end{definition}
For our purposes we consider random integrands $W\colon\Omega\x\R^n\x\mathbb{M}^{n\times n}\to[0,+\infty)$  satisfying the following assumptions:
\begin{enumerate}[label*=$(w\arabic*)$]
\item\label{hyp:w1} $W$ is $(\T\otimes\mathcal B^n\otimes \mathcal B^{n\times n})$-measurable;
\item\label{hyp:w2} $W(\omega,\cdot,\cdot)\in\mathcal W$ for every $\omega\in\O$;
\item \label{hyp:w3} the map $M\mapsto W(\omega,x,M)$ is lower semicontinuous for every $\omega\in\O$ and every $x\in\R^n$. 
\end{enumerate}
Let $W$ be a random integrand satisfying \ref{hyp:w1}-\ref{hyp:w3} and $\delta_k\searrow0$. We consider  the family of functionals $F_k(\omega)\colon L^0(U;\R^n)\to[0,+\infty]$ defined as
\begin{equation}\label{def:F_k_random}
	F_k(\omega)(u)\defas\frac1{\e_k}\int_Uf\bigg(\e_kW\Big(\omega,\frac{\cdot}{\delta_k},e(u)\Big)*\rho_k(x)\bigg)\dx\,,
\end{equation}
if $u\in W^{1,p}(U;\R^n)$, and extended to $+\infty$ otherwise. 
Let also for $\omega\in\O$ and $A\in\A$
\begin{equation}\label{min-prob-random}
	\m_\omega(u_M,A)\defas\inf\left\{\int_AW(\omega,x,e(v))\dx\colon v\in W^{1,p}(A;\R^n),\ v=u\ \text{near}\ \partial A
	\right\}\,.
\end{equation}
We now state the main theorem of this section. 
\begin{theorem}[Stochastic homogenisation]\label{thm:stoch-hom}
Let $W$ be a random integrand satisfying \ref{hyp:w1}-\ref{hyp:w3}. Assume moreover $W$ is stationary with respect to a group $(\tau_z)_{z\in\Z^n}$ of $P$-preserving transformations on $(\O,\T,P)$. For every $\omega\in\O$ let $F_k(\omega)$ be as in \eqref{def:F_k_random} and $\m_\omega$ be as in \eqref{min-prob-random}. Then there exists $\O'\in\T$, with $P(\O')=1$ such that for every $\omega\in\O'$, $x\in\R^n$, $M\in\mathbb{M}^{n\times n}$ the limit
\begin{equation}\label{eq:lim-stoch-hom}
	\lim_{t\to+\infty}\frac{\m_\omega(u_M,Q_t(tx))}{t^n}=	\lim_{t\to+\infty}\frac{\m_\omega(u_M,Q_t(0))}{t^n}=:W_{\rm hom}(\omega,M)
\end{equation}
exists and is independent of $x$. The function $W_{\rm hom}\colon\O\times \mathbb{M}^{n\times n}\to[0,+\infty)$ is $(\T\otimes \mathcal B^{n\times n})$-measurable. Moreover, for every $\omega\in\O'$ the functionals $F_k(\omega)$ $\Gamma$-converge in measure to the functional $F_{\rm hom}(\omega)\colon L^0(U;\R^n)\to[0,+\infty]$ given by 
\begin{equation*}
	F_{\rm hom}(\omega)(u)\defas\begin{cases}
		\displaystyle\	\alpha\int_U W_{\rm hom}(\omega,e(u))\dx+\beta \int_{J_u}\phi_\rho(\nu_u)\dHn&\text{if }u\in GSBD^p(U)\,,\\
		+\infty&\text{otherwise}\,.
	\end{cases}
\end{equation*}
If, in addition, $W$ is ergodic, then $W_{\rm hom}$ is independent of $\omega$ and 
\begin{equation}\label{eq:lim-stoch-hom2}
	W_{\rm hom}(M)=\lim_{t\to+\infty}\frac{1}{t^n}\int_\O\m_\omega(u_M,Q_t(0))\dP(\omega)\,,
\end{equation}
and thus $F_{\rm hom}$ is deterministic.
\end{theorem}
The rest of this section is devoted to prove Theorem \ref{thm:stoch-hom}. This will be done in a number of steps. 
\begin{proposition}\label{prop:subadditive-proc} Let $W$ be a stationary random integrand satisfying \ref{hyp:w1}-\ref{hyp:w3} and let $\m_\omega$ be as in \eqref{min-prob-random}.  Then for every $M\in\mathbb{M}^{n\times n}$ the function $\mu_M\colon\O\times\I_n\to\R$ given by 
$
	\mu_M(\omega,A)\defas \m_\omega(u_M,A)
$
defines a subadditive process on $(\O,\T,P)$.
\end{proposition}
\begin{proof} Let $M\in\mathbb{M}^{n\times s}$ be fixed.  Then we need to show that $\mu_M$  satisfies properties \ref{subad:meas}--\ref{subad:bound}.\\	
\textit{Step 1: measurability}. 
Let $A\in \I_n$ be fixed. For $N\in\N$ let
\begin{equation*}
W^N(\omega,x,M)\defas \inf_{\xi\in\mathbb{M}^{n\times n}}\{W(\omega,x,\xi)+N|\xi-M|\}
\end{equation*}
be the Moreau-Yosida regularisation of $M\mapsto W(\omega,x,M)$ which is $N$-Lipschitz. Let also $F^N(\omega)\colon W^{1,p}(A)\to[0,+\infty)$ be defined as
\begin{equation*}
F^N(\omega)(u)\defas\int_AW^N(\omega,x,e(u))\dx\,.
\end{equation*}
Arguing as in the proof of \cite[Lemma C.1.]{RufRuf} it can be shown that $(\omega,u)\mapsto F^N(\omega)(u)$ is $\T\otimes \mathcal{B}(W^{1,p}(A))$-measurable. By \ref{hyp:w3} $W^N\nearrow W$ pointwise, and in particular $F^N(\omega)(u)$ converges to $\int_AW(\omega,x,e(u))\dx$ pointwise. As a consequence $(\omega,u)\mapsto \int_AW(\omega,x,e(u))\dx$ is also $\T\otimes \mathcal{B}(W^{1,p}(A))$-measurable.  Now we note that $F(\omega)(u_M)<+\infty$. This together with \ref{hyp:w3} and \cite[Lemma C.2.]{RufRuf} imply that $\omega\mapsto\mu_M(u_M,A)$ is $\T$-measurable.
\\
 \textit{Step 2: covariance}. Let $\omega\in\O$, $A\in\I_n$ and $z\in\Z^n$ be fixed. Let $v\in W^{1,p}(A+z)$. Then $v=u_M$ near $\partial (A+z)$ if and only if $v(\cdot+z)=u_M$ near $\partial A$; moreover the stationarity of $W$ yields
\begin{equation*}
\int_{A+z}W(\omega,x,e(v))\dx=\int_{A}W(\omega,x+z,e(v(\cdot+z)))\dx=\int_A W(\tau_z(\omega),x,e(v(\cdot+z)))\dx\,.
\end{equation*}
Hence we have that 
\begin{equation*}
\mu_M(\omega,A+z)=\mu_M(\tau_z(\omega) ,A)\,.
\end{equation*}
 \textit{Step 3: subadditivity}. Let $\omega\in\O$ and $A\in\I_n$ be fixed. Let $\{A_1,...,A_N\}\subset\I_n$ be pairwise disjoint such that $A=\cup_{i=1}^NA_i$.
For each $i\in\{1,...,N\}$ we take $v_i\in W^{1,p}(A_i)$ with $v_i=u_M$ near $\partial A_i$. Let $v\in W^{1,p}(A)$ be given by $v:=v_i$ in $A_i$. Then clearly $v=u_M$ near $\partial A$ and 
\begin{equation*}
	\mu_M(\omega,A)\le \int_AW(\omega,x,e(v))\dx=\sum_{i=1}^N\int_{A_i}W(\omega,x,e(v_i))\dx\,.
\end{equation*}
By the arbitrariness of $v_i$ we then conclude 
\begin{equation*}
	\mu_M(\omega,A)
	\le \sum_{i=1}^N	\mu_M(\omega,A_i)
	\,.
\end{equation*}
\textit{Step 4: boundedness}. Let $\omega\in\O$ and $A\in\I_n$ be fixed. 
Then by \ref{hyp:growth-W} we get
\begin{equation*}
\mu_M(\omega,A)\le \int_AW(\omega,x,e(u_M))\dx\le c_2(|M+M^T|^p+1)\L^n(A)\,.
\end{equation*}
\end{proof}
With the help of Proposition \ref{prop:subadditive-proc} we can now prove the main result of this section. 
\begin{proof}[Proof of Theorem \ref{thm:stoch-hom}]
	By Theorem \ref{thm:det-hom} the $\Gamma$-convergence of $F_k(\omega)$ follows if we  show that there exists a set $\O'\in\T$ of full probability such that \eqref{eq:lim-stoch-hom} holds true for all $\omega\in \O'$.\\
	Let $M\in\mathbb{M}^{n\times n}$ be fixed.  By Proposition \ref{prop:subadditive-proc} and the Subadditive Ergodic Theorem \cite[Theorem 2.4]{AK81} we deduce the existence of $\O_M\in\T$ with $P(\O_M)=1$ and of a $\T$-measurable function $\phi_M\colon\O\to[0,+\infty)$ such that
	\begin{equation*}
	\phi_M(\omega)=\lim_{t\to+\infty}\frac{\m_\omega(u_M,t\tilde Q)}{|t\tilde Q|}\,,
	\end{equation*}
for all $\omega\in\O_M$ and all cubes $\tilde Q\subset\R^n$. We set
\begin{equation}\label{def:Omega}
	\O'\defas\bigcap_{M\in\Q^{n\times n}}\O_M\,,
\end{equation}
which satisfies $P(\O')=1$, and let $W_{\rm hom}\colon\O\times\mathbb{M}^{n\times n}\to[0,+\infty)$ be given by
\begin{equation*}
	W_{\rm hom}(\omega,M)\defas\limsup_{t\to+\infty}
	\frac{\m_\omega(u_M, Q_t(0))}{t^n}\,.
\end{equation*}
Clearly  $W_{\rm hom}(\omega,M)=\phi_M(\omega)$ for all $\omega\in\O'$, $M\in\Q^{n\times n}$. 
Let now $\omega\in \O'$, $t>0$, $\eta\in (0,1)$ and $\tilde Q$ (cube centred at $x$) be fixed. Choose $M_1\in\Q^{n\times n}$ and $M_2\in\mathbb{M}^{n\times n}$. Choose also $u\in W^{1,p}(t\tilde Q;\R^n)$ with  $u=u_{M_2}$ near $\partial (t\tilde Q)$ and
\begin{equation*}
\int_{t\tilde Q}W(\omega,x,e(u))\dx\le \m_\omega(u_{M_2},t\tilde Q)+\eta\,.
\end{equation*}
We extend $u$ with $u_{M_2}$ in $\R^n\setminus A$ without relabelling it. Next we take $\varphi_\eta\in C^\infty_c(\R^n,[0,1])$ with $\varphi_\eta\equiv 1$ on $t\tilde Q$, $\varphi_\eta\equiv0$ on $\R^n\setminus (1+\eta)t\tilde Q$ and $\|\nabla \varphi_\eta\|_{L^\infty(\R^n)}\le C/\eta$, and set
\begin{equation*}
v:=u\varphi_\eta+u_{M_1}(1-\varphi_\eta)\,,
\end{equation*}
which is admissible in the definition of $\m_\omega(u_{M_1}, (1+\eta)t\tilde Q)$.
By \ref{hyp:growth-W} it holds 
\begin{equation}\label{eq:stoch1}
	\begin{split}
	 \int_{(1+\eta)t\tilde Q}W(\omega,x,e(v))\dx
		\le \int_{t\tilde Q}W(\omega,x,e(u))\dx
		+ {c_2}\int_{(1+\eta)t\tilde Q\setminus t\tilde Q}(|\nabla v+\nabla v^T|^p+1)\dx\,.
	\end{split}
\end{equation}
As $v=u_{M_2}\varphi_\eta+u_{M_1}(1-\varphi_\eta)$ in $(1+\eta)t\tilde Q\setminus t\tilde Q$ we get
\begin{equation}\label{eq:stoch2}
	\begin{split}
|\nabla v+\nabla v^T|^p&\le  C\Big(|M_1+M_1^T|^p  + |M_2+M_2^T|^p+\frac1{\eta^p}|{M_1}-{M_2}|^p|x|^p \Big)\\
&\le  C\Big(|M_1+M_1^T|^p  + |M_2+M_2^T|^p+\frac1{\eta^p}|{M_1}-{M_2}|^p(1+\eta)^p \Big)\,.
	\end{split}
\end{equation}
Combining \eqref{eq:stoch1} with \eqref{eq:stoch2} we obtain
\begin{equation*}
	\begin{split}
\m_\omega(u_{M_1},& (1+\eta)t\tilde Q)\le \m_\omega(u_{M_2},t\tilde Q)\\&+ C\Big(|M_1+M_1^T|^p  + |M_2+M_2^T|^p+\frac1{\eta^p}|{M_1}-{M_2}|^p+1 \Big)((1+\eta)^n-1)t^n|\tilde Q|\,.
	\end{split}
\end{equation*}
Rescaling by $(1+\eta)^nt^n$ and passing to the limit as $t\to+\infty$ in the above inequality we find
\begin{equation}\label{eq:stoch3}
	\begin{split}
		W_{\rm hom}(\omega,M_1)&\le \frac{1}{(1+\eta)^n}\liminf_{t\to+\infty}\frac{\m_\omega(u_{M_2}t\tilde Q)}{t^n|\tilde Q|}\\
		&+C\Big(|M_1+M_1^T|^p  + |M_2+M_2^T|^p+\frac1{\eta^p}|{M_1}-{M_2}|^p+1 \Big)\frac{((1+\eta)^n-1)}{(1+\eta)^n}\,.
	\end{split}
\end{equation}
By inverting the role of $M_1$ and $M_2$ we can similarly find
\begin{equation}\label{eq:stoch4}
	\begin{split}
	\frac{1}{(1+\eta)^n} \limsup_{t\to+\infty}&\frac{\m_\omega(u_{M_2},t\tilde Q)}{t^n|\tilde Q|}	\le 	W_{\rm hom}(\omega,M_1)\\
		&+C\Big(|M_1+M_1^T|^p  + |M_2+M_2^T|^p+\frac1{\eta^p}|{M_1}-{M_2}|^p+1 \Big)\frac{((1+\eta)^n-1)}{(1+\eta)^n}\,.
	\end{split}
\end{equation}
Thus by the arbitrariness of $M_1$ we can choose a sequence $(M^j_1)\subset \Q^{n\times n}$ converging to $M_2$ such that by combining \eqref{eq:stoch3} and \eqref{eq:stoch4} we obtain
\begin{equation*}
\limsup_{t\to+\infty}\frac{\m_\omega(u_{M_2},t\tilde Q)}{t^n|\tilde Q|}\le 
\liminf_{t\to+\infty}\frac{\m_\omega(u_{M_2},t\tilde Q)}{t^n|\tilde Q|}+ C\Big( |M_2+M_2^T|^p+1\Big)((1+\eta)^n-1)\,.
\end{equation*}
Eventually by letting $\eta\to0$ we find that for all $\omega\in\O'$ there exists the limit 
\begin{equation}
	\lim_{t\to+\infty}\frac{\m_\omega(u_{M_2},t\tilde Q)}{t^n|\tilde Q|}=
	\limsup_{t\to+\infty}\frac{\m_\omega(u_{M_2},Q_t(0))}{t^n}=W_{\rm hom}(\omega,M_2)\,,
\end{equation}
and is independent of $x$.
Eventually if $W$ is ergodic then \cite[Theorem 2.4]{AK81} ensures that $W_{\rm hom}$ does not depend on $\omega$. Moreover by Birkhoff ergodic Theorem we deduce \eqref{eq:lim-stoch-hom2}.
\end{proof}
\section*{Acknowledgments}
\noindent
The work of FS was partially supported by the project \emph{Variational methods for stationary and evolution problems with singularities and interfaces} PRIN 2017 (2017BTM7SN) financed by the Italian Ministry of Education, University, and Research and by the project Starplus 2020 Unina Linea 1 \emph{New challenges in the variational modeling of continuum mechanics} from the University of Naples ``Federico II'' and Compagnia di San Paolo (CUP: E65F20001630003).
He is also member of the GNAMPA group of INdAM.

\setcounter{section}{0}
\renewcommand{\thesection}{A.\arabic{section}}
\appendix 
\section{A remark on the non-local approximation of free-discontinuity problems in $GSBV$} \label{sec:appendix}
\setcounter{theorem}{0}
\setcounter{equation}{0}
\renewcommand{\theequation}{A.\arabic{equation}}
\renewcommand{\thetheorem}{A.\arabic{theorem}}
This Appendix is devoted to the statement of a $\Gamma$-convergence Theorem for non-local functionals depending on the full deformation gradient $\nabla u$. The result we are going to state has actually been proved in \cite[Theorem 3.2]{Cortesani}, under an additional technical assumption, the so called {\it stable $\gamma$-convergence} of the functionals
\begin{equation}\label{Efull}
	\tilde{E}_k(u,A)\defas\begin{cases}
		\displaystyle\int_A  W_k(x,\nabla u)
		\dx& \text{ if } u\in W^{1,p}(A;\R^n)\,,\\
		+\infty&\text{ otherwise }\,.
	\end{cases}
\end{equation}
This assumption, stated in \cite[Definition 7.2]{Cortesani} is stronger than simple $\Gamma$-convergence, and introduces a limitation to the class of functionals to which the theorem applies, although relevant examples fulfilling this condition can be readily provided (see  \cite[Examples 7.3-7.5]{Cortesani}). Actually, the inspection of the proof of Proposition \ref{blow-up}, which can be clearly adapted to the $GSBV$ setting, shows that it is not needed. For the reader's convenience we give a precise statement of the result, after recalling the structural assumptions on the non-local approximation energies under which it is formulated.

The functions $W_k$ are assumed to satisfy \ref{hyp:conv+lsc-W}--\ref{hyp:min-W}, together with
\begin{enumerate}[label*=$(W4^\prime)$]
\item\label{hyp:growth} for every $x\in \R^n$ and every $M\in \mathbb{M}^{n\times n}$
	\begin{equation*}
		c_1|M|^p\le W_k(x,M)\le c_2(|M|^p+1)\,.
	\end{equation*}
\end{enumerate}
We will denote with $\tilde E$ the $\Gamma$-limit with respect to the convergence in measure of the functionals $\tilde E_k$ in \eqref{Efull}, given by (see \cite[Theorem 20.4]{DM})
\[
	\tilde{E}(u,A)\defas\begin{cases}
		\displaystyle\int_A  W(x,\nabla u)
		\dx& \text{ if } u\in W^{1,p}(A;\R^n)\,,\\
		+\infty&\text{ otherwise }\,.
	\end{cases}
\]
where, for every $x\in u$ and every $M\in \mathbb{M}^{n\times n}$
\begin{equation}\label{Wfull}
W(x,M)=W'(x, M)=W''(x, M)\,.
\end{equation}
Above, $W'$ and $W''$ are defined in \eqref{eq:cell-formula-inf}, and  \eqref{eq:cell-formula-sup}, respectively, provided that $E_k$ is replaced by $\tilde E_k$.
We then consider the non-local functionals 
\begin{equation}\label{Ffull}
	\tilde F_k(u)\defas\begin{cases}
		\displaystyle\frac1{\e_k}\int_U f_k\Big(\e_kW_k(\cdot,\nabla u)*\rho_k(x)
	\Big)\dx& \text{ if } u\in W^{1,p}(U;\R^n)\,,\\
		+\infty&\text{ otherwise }.
	\end{cases}
\end{equation}
where $\rho_k$ are as in Section \ref{subs:setting}, while $f_k\colon[0, +\infty)\to [0, +\infty)$ are concave and satisfy
\begin{equation}\label{cortesanifk}
a_1 t \wedge b_1 \le f_k(t) \le b_2
\end{equation}
for suitable uniform constants $a_1$, $b_1$, $b_2 >0$. We then have the following theorem.

\begin{theorem}\label{cortesani-improved}
Assume  \ref{hyp:conv+lsc-W}, \ref{hyp:min-W}, and \ref{hyp:growth}. Consider a sequence of concave functions $f_k$ as in \eqref{cortesanifk} and convolution kernels $\rho_k$ as in Section  \ref{subs:setting}. Let the functionals $\tilde F_k$ be given by \eqref{Ffull}. Finally, assume that
\begin{equation}\label{cortesanifk+}
\alpha_k t \wedge b_1 \le f_k(t) \le b_2 \quad \mbox{with}\quad \lim_{k\to +\infty}\alpha_k-f^\prime_k(0)=0\,.
\end{equation}
Then $\tilde F_k$ $\Gamma$-converge, with respect to the convergence in measure, to a functional of the form
\[
\alpha \int_U  W(x,\nabla u)
		\dx+ \int_{J_u}\varphi(x, [u], \nu_u)\dHn
\]
where $W$ is given by \eqref{Wfull}, $\alpha=\liminf f_k^\prime(0)$, and $\varphi$ is a suitable Carath\'eodory integrand.
\end{theorem}

\begin{proof} By \cite[Theorem 3.1]{Cortesani} we have that the $\Gamma$-limit of $\tilde F_k$ is an integral functional of the form
\[
\int_U  W_\infty(x,\nabla u)
		\dx+ \int_{J_u}\varphi(x, [u], \nu_u)\dHn\,.
\]
For $W'$ and $W''$ as in \eqref{eq:cell-formula-inf}, and  \eqref{eq:cell-formula-sup}, respectively, one has only to show that $W_\infty\le \alpha W''$ and $W_\infty\ge \alpha W'$. The first inequality is actually already proved in \cite[Proposition 7.1]{Cortesani}. As for the second, notice under assumption \eqref{cortesanifk+} and taking into account Remark \ref{cortesani}, it can be recovered by exactly following the argument of Proposition \ref{blow-up}, provided one is willing to replace each occurrence of $e(u)$ with $\nabla u$.
\end{proof}

\end{document}